\documentclass[11pt]{amsart}
\usepackage{graphicx}
\usepackage{epstopdf}

\usepackage{amsmath}
\usepackage{xcolor}
\usepackage{float}
\usepackage{mathtools}
\usepackage{amsthm}
\usepackage{cite}
\usepackage{stix}
\usepackage{amsfonts}
\usepackage{xparse}
\usepackage{setspace}
\usepackage{adjustbox}
\usepackage{subcaption}
\usepackage{braket}
\usepackage{tikzit}
\usepackage{blindtext}
\usepackage[breaklinks]{hyperref}
\usepackage{bookmark}
\usepackage[toc,page]{appendix}
\usepackage[font={small}]{caption}
\usepackage{tikzcode}
\usepackage{thm-restate}

\DeclareEmphSequence{\bfseries\itshape}
\newcommand{\bmb}{\left( \begin{array}{rr}}
\newcommand{\enm}{\end{array}\right)}

\newcommand{\cM}{\mathcal M}

\newcommand{\M}{{\mathcal M}}

\hypersetup{
    colorlinks,
    linkcolor={red!50!black},
    citecolor={blue!100!black},
    urlcolor={blue!80!black}
}

\newcommand{\C}{{\mathbb C}}
\newcommand{\Z}{{\mathbb Z}}

\newcommand{\bx}{{\mathbf x}}

\newcommand{\bJ}{{\mathbf J}}

\newcommand{\mcard}[2]{m_{<#1}(#2)}
\newcommand{\tmcard}[2]{m_{<#1}(#2)}
\newcommand{\compratio}{{\mathfrak R_{\C(q;t)}}}

\numberwithin{equation}{section}

\theoremstyle{definition}
\newtheorem{thm}{Theorem}[section]
\newtheorem{prop}[thm]{Proposition}
\newtheorem{defn}[thm]{Definition}
\newtheorem{lemma}[thm]{Lemma}

\newtheorem{remark}[thm]{Remark}

\title{Rank recursion for $q$-Whittaker and Macdonald operators}

\author{Philippe Di Francesco}
\address{Department of Mathematics, University of Illinois Urbana-Champaign, IL 61801, USA}
\email{philippe@illinois.edu}

\author{Hieu Trung Vu}
\address{Yau Mathematical Sciences Center, Tsinghua University, Haidian District, Beijing 100084, China}
\email{htrungvu@mail.tsinghua.edu.cn}

\date{\today}

\begin{document}

\begin{abstract}
In this paper, we introduce and prove a set of rank recurrence relations for $q$-Whittaker and Macdonald operators. We also derive an explicit expression for the  $k$-th power of the $q$-Whittaker operator in terms of the $q$-deformed binomial probability distribution, and we express the rank recurrence relations for the Macdonald operators in terms of the Cauchy determinant.
\end{abstract}

\maketitle
\tableofcontents

\section{Introduction}
The aim of this paper is to introduce a family of  rank recurrence relations for $q$-Whittaker and Macdonald operators, which are difference operators that act on symmetric functions over $\C(q,t)$. These operators have the $q$-Whittaker and Macdonald polynomials as their eigenfunctions. Introduced in the 1980s by Ian G. Macdonald, the Macdonald polynomial is a family of symmetric polynomials with coefficients that are rational functions in two variables $q$ and $t$, see the complete monograph by Macdonald \cite{macdonald1995symmetric}. They are a generalization of many well-known symmetric polynomials via specializing $q$ and $t$, such as Schur polynomials ($q=t$), Jack polynomials ($t=q^\alpha$, $q\to 1$), Hall-Littlewood polynomials ($q=0$), and $q$-Whittaker polynomials ($t=0$).  Macdonald polynomials are indexed by integer partitions $\lambda=(\lambda_1 \geq \lambda_2 \geq \cdots)$ denoted by $P_\lambda(x_1,\ldots,x_N;q,t)$. One of their descriptions is as eigenfunctions of the level-$r$ Macdonald operators on the ring of symmetric functions over $\C(q;t)$ with eigenvalues given by the $r$-th elementary symmetric functions $e_r$, (see the precise Definition \ref{defn: original Macdo} below).
\begin{equation*}
	\cM^r_{N} = \sum_{\substack{I \subset \{1, \dots, N\}\\ I \neq \emptyset}}^N \prod_{\substack{i \in I\\ j \notin I}}\frac{tx_i-x_j}{x_i-x_j}
	\Gamma_i,\qquad 
	\Gamma_i f(x_1,\ldots,x_N) = f(x_1,\ldots,qx_i,\ldots,x_N).
\end{equation*}
The eigenvalue of $P_\lambda$ under $\cM_{N}^{r}$ is given by
\begin{equation*}
	\cM_{N}^{r} P_\lambda = e_r(q^{\lambda_1}t^{N-1}, q^{\lambda_2}t^{N-2},\dots, q^{\lambda_N}) P_\lambda.
\end{equation*}
The Macdonald operators are known not only to possess rich algebraic structures, but also to appear in the context of integrable systems and interacting particle systems. For example, they can be expressed in terms of the double affine Hecke algebra (DAHA) generators \cite{Cherednik2005}, as $q$-moment generators for observables in interacting particle systems \cite{BC14}. They also have connections to discrete integrable systems and quantum cluster algebras \cite{DFK24,DiFrancescoKedem2019,DiFrancescoKedem2025} and can be represented as Hamiltonians of Ruijsennars-Schneider models \cite{Ruijsenaars-Schneider1986,Ruijsenaars1987}, which are a relativistic generalization of the integrable quantum Calogero-Moser-Sutherland systems. Each of these connections has been the subject of intensive research in recent decades, and we refer to the references therein for more details.

At the level of polynomials, Macdonald polynomials are known to possess recursive constructions similar to their Schur polynomials and Jack polynomials counterparts, known as \emph{the branching rule} or \emph{the Pieri rule}. See, for example, Chapter VI of \cite{macdonald1995symmetric} for the classical theory, \cite{Haglund2021,GarsiaHaiman1995, Okounkov1998, LassalleSchlosser2006,LascouxWarnaar2011,vanDiejenEmsiz2015,vanDiejenEmsiz2018, corteel2022multiline} for more recent developments, and \cite{Cherednik2005} for recursion through DAHA intertwiners. However, to the best of our knowledge, recursive constructions for the Macdonald operators themselves have not been well studied, except for some constructions of (interpolating) Macdonald operators at infinity; see \cite{NS14,Cuenca2018} or through shuffle product in \cite{DiFrancescoKedem2018,DiFrancescoKedem2019}. The main result of this paper is  a recursive construction for both $q$-Whittaker and Macdonald operators in Theorem \ref{thm:rank-recursion} and  Theorem \ref{conj: rank recursion}. The two recurrence relations read, respectively:
$$
D^{a}_{N}\left(x_1, \ldots, x_N\right)=\sum_{i=N-a}^N \prod_{\substack{j= N-a\\j \neq i}}^N \frac{x_j}{x_j-x_i} D^{a}_{N-1}\left(x_1, \ldots, x_{i-1}, x_{i+1}, \ldots, x_N\right),
$$
and
$$
\mathcal M^a_{N}(x_1,\dots, x_N)=\sum_{\emptyset\neq J \subseteq K}(-1)^{|J|-1}t^{|J|(|J|-1)/2}\prod_{\substack{j \in K \backslash J\\ i \in J}}\frac{tx_j-x_i}{x_j-x_i}\mathcal M^a_{N-|J|}(\{x_k, k \in [1, N]\backslash J\}),
$$
for any subset $K = \{k_0, \dots, k_{a}\} \subset \{1,\ldots,  N\}$ with $k_0<k_1<\dots<k_{a}$.

These recurrence relations also allow us to derive explicit expressions for the $k$-th power of the $q$-Whittaker and Macdonald operators in Theorem \ref{thm: q-Whittaker power} and Theorem \ref{eq: Macdonald level a power formula}, respectively. In particular, the $k$-th power of the $q$-Whittaker operator can be expressed in terms of a $q$-deformed binomial probability distribution, which was communicated to us by Leonid Petrov through his private discussion with Andrew Ahn \cite{AhnPetrov2024}. As a consequence, we also derive some commutator identities for the $k$-th power of the Macdonald operators with the operator of multiplication by the sum of variables, which generalizes those in Section 6 of \cite{BC15}. These relations, once found, will be useful in studying Macdonald processes and interacting particle systems for models with parameter symmetry \cite{PS24,Petrov2021}.

The paper is organized as follows. In Section \ref{sec:whittaker-rank-recursion}, we prove the rank recursion for the $q$-Whittaker operators using a Lagrange interpolation argument and provide an explicit expression for the $k$-th power of the $q$-Whittaker operators at the end of Section \ref{sec:whittaker-rank-recursion}. In Section \ref{sect:macdonald_recursion}, we prove the rank recursion for the Macdonald operators using a more involved argument on the coefficients of the $q$-shift operators, whose the technical details are in Appendix \ref{appendix1} and \ref{appendix2}. In Section 4, we generalize  some commutator identities that showed up before in the context of Macdonald processes and interacting particle systems \cite{BC14,BC15,BCGS16} using the power formula from Section \ref{sect:macdonald_recursion}. 

\textbf{Acknowledgement:}
The authors would like to thank Leonid Petrov for introducing  us to the relations \eqref{eq: q-Whittaker power} and \eqref{eq:three-term rank recursion} during the long program Geometry, Statistical Mechanics, and Integrability at Insti-
tute for Pure and Applied Mathematics, supported by the National Science Foundation (Grant No. DMS-1925919).


\section{Rank recursion for \texorpdfstring{$q$}{q}-Whittaker and Macdonald Operators}\label{sec:whittaker-rank-recursion}
We consider the (time-translated) $\mathrm{q}$-Whittaker operators of rank $\mathrm{N}$, acting on functions of the pairwise distinct variables $\left(x_1, \ldots, x_N\right)$, associated with $P=(x_1 x_2 \cdots x_N)^n$:
$$
D^{a, n}_{N}\left(x_1, \ldots \ldots x_N \right):=\sum_{\substack{I \subset  [N] \\ |I|=a}} \prod_{\substack{i \in I \\j \not \in I}} \frac{x_i}{x_i-x_j}\left(x_I\right)^n \Gamma_I
$$
where $x_I=\prod_{i \in I} x_i$, $\Gamma_I = \prod_{i \in I} \Gamma_i$, $[N]:=\{1,\dots, N\}$ and $\Gamma_i=q^{x_i \partial_{x_i}}$.

When $P=1$, the operators reduce to the standard $q$-Whittaker operators:
$$D^{a, 0}_{N}\left(x_1, \ldots, x_N\right)=\sum_{\substack{I \subset [N] \\ |I|=a}} \prod_{\substack{i \in I \\ j \not \in I}} \frac{x_i}{x_i-x_j} \Gamma_I.$$

Let us first give a motivating example for the rank recursion for the $q$-Whittaker operators. Consider the case $a=1, n=0$, we have: 
\begin{equation}\label{eq: affine relation}
D_N^{1,0}(x_1, \dots, x_N)=\frac{x_N}{x_N-x_{N-1}}D_{N-1}^{1,0}(x_1, \dots, x_{N-1})+\frac{x_{N-1}}{x_{N-1}-x_N}D_{N-1}^{1,0}(x_1, \dots, x_{N-2}, x_N).
\end{equation}
One can check this relation by considering the coefficients of the $q$-shift operators $\Gamma_i$ on both sides. For example, for fixed $i=0, \dots, N-2$, we have on the r.h.s. the coefficient of $\Gamma_i$ is given by
$$\frac{x_N}{x_N-x_{N-1}}\prod_{j \neq i, N-1} \frac{x_i}{x_i-x_j}+\frac{x_{N-1}}{x_{N-1}-x_N}\prod_{j \neq i, N} \frac{x_i}{x_i-x_j}=\prod_{j \neq i} \frac{x_i}{x_i-x_j},$$

\begin{thm}\label{thm:rank-recursion}
We have the recursion relation:
$$
D^{a, n}_{N}\left(x_1, \ldots, x_N\right)=\sum_{i=N-a}^N \prod_{\substack{j= N-a\\j \neq i}}^N \frac{x_j}{x_j-x_i} D^{a, n}_{N-1}\left(x_1, \ldots, x_{i-1}, x_{i+1}, \ldots, x_N\right).
$$
\end{thm}

\begin{proof}
	Collecting the coefficients of $\Gamma_I$ on both sides, the identity boils down to the relation:
\begin{equation}\label{eqn:qwhitt_coeff}
\prod_{\substack{k\in I \\ \ell \not \in I}} \frac{x_k}{x_k-x_\ell}=\sum_{\substack{i=N-a}}^{N} \prod_{\substack{j=N-a\\j\neq i}}^N \frac{x_j}{x_j-x_i} \prod_{\substack{k \in I_i\\ \ell \notin I_i, \ell \in \{1,\dots,N\}\backslash \{i\}}} \frac{x_k}{x_k-x_\ell}.
\end{equation}
where we define $I_i=I \backslash\{i\}$ if $i \in I, I_i=I$ otherwise. However, as only $i \notin I$ contribute to the r.h.s. (otherwise $D_{a, n}^{(N-1)}\left(x_1, \ldots , x_{i-1}, x_{i+1}, \ldots, x_N\right)$ would not have a $\Gamma_I$ term), we may rewrite this as:
$$
\prod_{\substack{k\in I \\ \ell \not \in I}} \frac{x_k}{x_k-x_\ell}=\sum_{\substack{i=N-a \\ i \notin I}}^{N} \prod_{\substack{j=N-a\\j\neq i}}^N \frac{x_j}{x_j-x_i} \prod_{\substack{k \in I\\ \ell \notin I, \ell \neq i}} \frac{x_k}{x_k-x_\ell}.
$$
Multiplying both sides with $\displaystyle \prod_{\substack{k \in I \\ \ell \notin I}} \left(x_k - x_\ell\right)$, we have

$$
x_I=\sum_{\substack{i=N-a \\ i \notin I}}^N \prod_{\substack{j=N-a\\j \neq i}}^N \frac{x_j}{x_j-x_i} \prod_{k \in I}\left(x_k-x_i\right) 
$$

This latter identity can be viewed as an interpolation relation for the polynomial $P(x)=x_I$, by choosing interpolation points $x^*(i):=\left(x_k, k \in \{N-a,\dots, N\} \backslash\{i\}\right), i \in \{N-a,\dots, N\}$, at which the prescribed value is $P\left(x^*(i)\right)=x_{\{N-a,\dots, N\}\backslash\{i\}}$, namely
$$
P(x)=\sum_{i=N-a}^{N} P\left(x^*(i)\right) \frac{\prod_{k \in I}\left(x_k-x_i\right)}{\prod_{j \in\{N-a, \dots, N\} \backslash \{i\}} \left(x_j-x_i\right)}
$$

Upon relabeling $\left(x_i, i \in I\right)$ as $y_1, y_2, \ldots, y_a$ and $\left(x_k, k \in\{N-a, \dots, N\}\right)$ as $z_1, \ldots, z_{a+1}$, this reads:
$$
y_1 y_2 \cdots y_a=\sum_{i=1}^{a+1} \prod_{\substack{k=1 \\ k\neq i}}^{a+1} \frac{z_k}{z_k-z_i} \prod_{j=1}^a\left(y_j-z_i\right)
$$

Note that setting $y_a=z_{a+1}$ yields:
$$
y_1 \cdots y_{a-1}=\sum_{i=1}^a \prod_{\substack{k=1 \\ k \neq i}}^a \frac{z_k}{z_k-z_i} \prod_{j=1}^{a-1}\left(y_j-z_i\right)
$$
which is the same identity for $a \rightarrow a-1$. Analogous relations are found when setting $y_a=z_1, z_2, \ldots, z_a$ as well. By induction, the identity therefore boils down to the case $\mathrm{a}=1$, where:
$$
y_1=\frac{z_2}{z_2-z_1}\left(y_1-z_1\right)+\frac{z_1}{z_1-z_2}\left(y_1-z_2\right)
$$
which is the Lagrange interpolation polynomial for $P\left(y_1\right)=y_1$ using values at $z_1$ and $z_2$. The identity \eqref{eqn:qwhitt_coeff} then follows.
\end{proof}     

Define the following quantities: 
\begin{align*}
	D_N&:=D_N^{a=1;n=0}(x_1, \dots, x_N)=\sum_{i=1}^N \prod_{j \neq i} \frac{x_i}{x_i-x_j} \Gamma_i\\
	\widetilde D_N&:=D_N^{a=1;n=0}(x_1, \dots, x_{N-1},x_{N+1})\\
	\varphi_{q, \alpha}(j|k)&:=\alpha^j(\alpha;q)_{k-j}\frac{(q;q)_k}{(q;q)_{k-j}(q;q)_j}
\end{align*}
where $j,k$ are non-negative integers, $\alpha$ is a parameter and $(a ; q )_n	$ denotes the $q$-Pochhammer symbol
\[{\displaystyle (a;q)_{n}=\prod _{k=0}^{n-1}(1-aq^{k})=(1-a)(1-aq)(1-aq^{2})\cdots (1-aq^{n-1}),}\]
with the convention $(a;q)_0=(q;q)_0:=1$. For integers $0\le j\le k$, define the $q$-binomial coefficient as
\[
\binom{k}{j}_q:=\frac{(q;q)_k}{(q;q)_j(q;q)_{k-j}},
\qquad
\varphi_{q,\alpha}(j\mid k):=\alpha^j\,(\alpha;q)_{k-j}\,\binom{k}{j}_q.
\]
Extend by $\varphi_{q,\alpha}(j\mid k):=0$ if $j<0$ or $j>k$, then for every integer $k\ge 0$ and every integer $j$,
\begin{equation}\label{eq:pascal}
\varphi_{q,\alpha}(j\mid k{+}1)
=
\bigl(1-\alpha q^{k-j}\bigr)\,\varphi_{q,\alpha}(j\mid k)
+
\alpha\,q^{k+1-j}\,\varphi_{q,\alpha}(j{-}1\mid k),
\end{equation}
and moreover the coefficients are normalized:
\begin{equation}\label{eq:sum1}
\sum_{j=0}^{k}\varphi_{q,\alpha}(j\mid k)=1\qquad\text{for all }k\ge 0.
\end{equation}
These properties of the function $\varphi_{q,\alpha}(j\mid k)$ are written in several texts, for example, see  \cite{Povolotsky2013}. The following identity was modified from a relations shared with us by Leonid Petrov through his private communication with Andrew Ahn \cite{AhnPetrov2024}.

\begin{thm}
	\label{thm: q-Whittaker power}
	For any $k >0$, 
	\begin{equation}
		\label{eq: q-Whittaker power}
		(\widetilde D_N)^k
=\sum_{j=0}^{k}\varphi_{q^{-1},\frac{x_N}{x_{N+1}}}(j\mid k)\;(D_N)^j(D_{N+1})^{k-j}
	\end{equation}
\end{thm}
 
Before proceeeding to the proof, let us state several useful identities and also the strategy for proving Theorem \ref{thm: q-Whittaker power}. We will show that the action from both side of \eqref{eq: q-Whittaker power} on the formal power series $\Phi(y)$ defined by 
$$
\Phi(y):=\prod_{i=1}^{N+1}(x_i y;q)_\infty
=\sum_{m\in\mathbb Z_{\ge0}^{N+1}}
  \prod_{i=1}^{N+1}\frac{(-1)^{m_i}q^{\binom{m_i}{2}}}{(q;q)_{m_i}}
  \,x^m y^m,
\qquad x^m:=\prod_{i=1}^{N+1}x_i^{m_i}.
$$
coincide. Note that this is sufficient to prove Theorem \ref{thm: q-Whittaker power} for the following reason. Define
$$\begin{aligned} \mathcal O:&= (\widetilde D_N)^k - \sum_{j=0}^{k}\varphi_{q^{-1},\frac{x_N}{x_{N+1}}}(j\mid k)\;(D_N)^j(D_{N+1})^{k-j}\\
	&=\sum_{I \in \Z^{N+1}_{\ge 0}} c_I(x)\Gamma^I, \quad \Gamma^I:=\prod_{i=1}^{N+1}\Gamma_i^{I_i}, \quad \Gamma^Ix^m=q^{\sum_i I_i m_i}x^m.
\end{aligned}
	$$ 
	for some rational function $c_I(x)$. If the action of from both sides of \eqref{eq: q-Whittaker power} on $\Phi(y)$ coincide, because of $q$-Whittaker operators only acts on $x$ variables, we must have $\mathcal O$ annihilates $\Phi(y)$ for all $y$, i.e. $\mathcal O \Phi(y) = 0$. Thus, since the coefficient of $x^m$ in $\Phi(y)$ is nonzero for generic $q$,
$$0=\mathcal O x^m=x^m \sum_I c_I(x) q^{\langle I, m\rangle},$$
which can only happen if $c_I(x) = 0$ for all $I$, i.e. $\mathcal O = 0$. Let us now proceed with the proof. 

Define the $q$-Whittaker operator on the subset of variables $S \subseteq \{1, \dots, N\}$ to be 
\begin{equation}
D_S:=\sum_{i\in S}\prod_{\substack{j\in S\\ j\ne i}}\frac{x_i}{x_i-x_j}\Gamma_i,
\end{equation}
with the convention that $D_N^0=D_{N+1}^0=\mathrm{id}$.
\begin{lemma}
	\label{lemma: partial action q-Whittaker}
	$$(D_S)^k\,\Phi(y)=\left(\prod_{r=0}^{k-1}\prod_{j \in S}\frac{1}{1-q^rx_jy}\right)\Phi(y).$$
\end{lemma}
\begin{proof}
	We proceed by induction on $k$. The base case $k=0$ is trivial since $(D_S)^0=\mathrm{id}$. Assume the identity holds for some $k\ge 0$. Then for $k+1$, we have
	\begin{equation}
		\label{eq: partial induction}
		\begin{aligned}
	&\quad (D_S)^{k+1}\,\Phi(y) = \sum_{i \in S}\prod_{j \neq i}\frac{x_i}{x_i-x_j}\,\Gamma_i\left(\prod_{r=0}^{k-1}\prod_{j \in S}\frac{1}{1-q^rx_jy}\right)\Gamma_i\Phi(y)\\
	&=\sum_{i \in S}\left(\prod_{j \neq i}\frac{x_i}{x_i-x_j}\right)\frac{1}{1-q^kx_iy}\prod_{r=0}^{k-1}\prod_{j \in S}\frac{1}{1-q^rx_jy}\Phi(y),\quad \textrm{since}\\
	&\Gamma_i\Phi(y)=\frac{1}{1-x_iy}\Phi(y), \qquad \frac{\Gamma_i\left(\prod_{r=0}^{k-1}\prod_{j \in S}\frac{1}{1-q^rx_jy}\right)}{\prod_{r=0}^{k-1}\prod_{j \in S}\frac{1}{1-q^rx_jy}}=\prod_{r=0}^{k-1}\frac{1-q^rx_iy}{1-q^{r+1}x_iy}=\frac{1}{1-q^kx_iy}.
		\end{aligned}
	\end{equation}
	Assuming that $|S|=n$ and let us consider the following partial fraction decomposition of the following rational function in some formal variables $z$:
	$$\frac{z^{n-1}}{\prod_{j \in S}(z-x_j)} = \sum_{i \in S} \frac{A_i}{z-x_i}$$
	Then, since the degree of the numerator is $n-1$ and the denominator has $n$ distinct roots, the coefficients $A_i$ are uniquely determined by $$A_i=\lim_{z\to x_i}(z-x_i)\frac{z^{n-1}}{\prod_{j \in S}(z-x_j)} =\prod_{j \neq i}\frac{x_i}{x_i-x_j}.$$
	Substituting $z = 1/q^ky$, we get
	$$\prod_{j \in S}\frac{1}{1-q^kx_jy} = \sum_{i \in S} \left(\prod_{j \neq i}\frac{x_i}{x_i-x_j}\right)\frac{1}{1-q^kx_iy}.$$
	Combining with \eqref{eq: partial induction}, we obtain
	\[
	(D_S)^{k+1}\,\Phi(y) = \left(\prod_{r=0}^{k}\prod_{j \in S}\frac{1}{1-q^rx_jy}\right)\Phi(y),
	\]
	which completes the induction step.
\end{proof}

This lemma provides the following relations for the power $k$ of $\widetilde D_N$
\begin{prop}
	\label{prop: D power raw}
\begin{equation}
	\label{eq: D power raw}
	\begin{aligned}
	(\widetilde D_N)^k \Phi(y) &= \prod_{t=0}^{k-1}\left(\prod_{i=1}^{N-1}\frac{1}{1-q^tx_iy}\right)
\left(\prod_{t=0}^{k-1}\frac{1}{1-q^tx_{N+1}y}\right)\Phi(y),\\
(D_N)^j(D_{N+1})^m\Phi(y)
&=
\left(\prod_{r=0}^{m-1}\prod_{i=1}^{N+1}\frac{1}{1-q^rx_iy}\right)
\left(\prod_{s=0}^{j-1}\prod_{i=1}^{N}\frac{1}{1-q^{m+s}x_iy}\right)\Phi(y), \quad j,m \geq 0
	\end{aligned}
\end{equation}
\end{prop}
\begin{proof}
	The first equality follows directly from Lemma \ref{lemma: partial action q-Whittaker} by setting $S = \{1, \dots, N-1, N+1\}$. The second equality can be obtained by the same induction procedure on $j$. Suppose the
formula holds for $j$ and write its scalar prefactor as $G_j\Phi$. For
$i\in [N]$, the factors of $G_j$ involving $x_i$ are $(1-q^rx_iy)^{-1}$ with $0\le r<m$ and 
$(1-q^{m+s}x_iy)^{-1}$ with $0\le s<j$; those  factors with $x_{N+1}$ are untouched by $\Gamma_i$. Hence,
$$
\frac{\Gamma_i(G_j)}{G_j}
=\prod_{r=0}^{m-1}\frac{1-q^rx_iy}{1-q^{r+1}x_iy}
 \prod_{s=0}^{j-1}\frac{1-q^{m+s}x_iy}{1-q^{m+s+1}x_iy}
=\prod_{r=0}^{m+j-1}\frac{1-q^rx_iy}{1-q^{r+1}x_iy},
$$ 
$$
\frac{\Gamma_i(G_j)\,\Gamma_i(\Phi)}{G_j\Phi}
=\frac{1}{1-x_iy}\prod_{r=0}^{m+j-1}\frac{1-q^rx_iy}{1-q^{r+1}x_iy}
=\frac{1}{1-q^{m+j}x_iy}.
$$
By a similar partial fraction argument as in the proof of Lemma \ref{lemma: partial action q-Whittaker} with $S = [N]$ and $z=1/(q^{m+j}y)$, one can show that 
\[D_N(G_j\Phi(y))=G_j\Phi(y)\sum_{i=1}^{N}\left(\prod_{\substack{\ell=1\\ \ell\neq i}}^{N}\frac{x_i}{x_i-x_\ell}\right)\frac{1}{1-q^{m+j}x_iy}=G_j\Phi(y)\prod_{i=1}^{N}\frac{1}{1-q^{m+j}x_iy}\]
which is the desired result for the induction step.
\end{proof}
By letting $m=k-j$ in the second equality of \eqref{eq: D power raw}, we obtain:
\begin{equation}
	\label{eq: r.h.s action}
	(D_N)^j(D_{N+1})^{k-j}\Phi(y)
=
\left(\prod_{t=0}^{k-1}\prod_{i=1}^{N}\frac{1}{1-q^tx_iy}\right)
\left(\prod_{r=0}^{k-j-1}\frac{1}{1-q^rx_{N+1}y}\right)\Phi(y).
\end{equation}
We are now ready to prove Theorem \ref{thm: q-Whittaker power}.

\begin{proof}[Proof of Theorem \ref{thm: q-Whittaker power}]
From Proposition \ref{prop: D power raw}, after removing all common scalar prefactor, it is sufficient to prove the following equality:
\begin{equation}
	\label{eq: last identity to show D power}
	\prod_{t=0}^{k-1}\frac{1}{1-q^tx_{N+1}y}
=
\left(\prod_{t=0}^{k-1}\frac{1}{1-q^tx_Ny}\right)
\sum_{j=0}^{k}\varphi_{q^{-1},x_N/x_{N+1}}(j\mid k)
\prod_{r=0}^{k-j-1}\frac{1}{1-q^rx_{N+1}y}.
\end{equation}
Set $z=x_{N+1}y$ and $\beta=x_N/x_{N+1}$, so that $x_Ny=\beta z$. Then the
required identity is
$$
\sum_{j=0}^{k}\varphi_{q^{-1},\beta}(j\mid k)\prod_{r=0}^{k-j-1}\frac{1}{1-q^rz}=
\sum_{j=0}^{k}\varphi_{q^{-1},\beta}(j\mid k)(z;q)_{k-j}^{-1}=
\prod_{t=0}^{k-1}\frac{1-\beta q^tz}{1-q^tz}=\frac{(\beta z;q)_k}{(z;q)_k}.
$$
Cancelling the denominator $(z;q)_k$ from both sides, 
$$
\sum_{j=0}^{k}\varphi_{q^{-1},\beta}(j\mid k)(zq^{k-j};q)_{j}=(\beta z;q)_k
$$
Recall the Gaussian $q$-Binomial identity: 
$$\prod_{t=0}^{n-1}(1+q^tz)=\sum_{r=0}^{n}q^{\binom r2}\binom nr_q z^r.$$
It suffices to show for $0 \leq r\leq k$, the coefficent of $z^r$ on both sides of the identity matches. On the l.h.s, expanding according to the $q$-Binomial identity, the coefficient of $z^r$ is 
$$\begin{aligned}\sum_{j=r}^{k} \varphi_{q^{-1},\beta}(j\mid k) (-1)^r q^{\binom r2} \binom{j}{r}_q q^{r(k-j)}&=\sum_{j=r}^{k} \beta^j(\beta;q^{-1})_{k-j}\binom{k}{j}_{q^{-1}} (-1)^r q^{\binom r2} \binom{j}{r}_q q^{r(k-j)}\\
	&=\beta^r\binom kr_q(-1)^r q^{\binom r2} \sum_{j=r}^{k} \beta^{j-r}(\beta;q^{-1})_{k-j} q^{-(j-r)(k-j)} \binom{k-r}{j-r}_q\\
	&=\beta^r\binom kr_q(-1)^r q^{\binom r2}.
\end{aligned}$$
where the last equality uses the fact the the remain sum can be written as the normalized $q$-Binomial $\sum_{a=0}^{b} (\beta;q^{-1})_{b-a} \beta^a\binom{b}{a}_{q^{-1}} = 1$ upon changing the variables $j-r \mapsto a$ and $k-r \mapsto b$. This is exactly the coefficient of $z^r$ on the r.h.s which completes the proof.

\end{proof}

\section{Rank Recursion for Macdonald Operators}
\label{sect:macdonald_recursion}
\begin{defn}
\label{defn: original Macdo}
The \emph{original Macdonald operators} are defined as:
\begin{equation}
	\label{eq:original Macdo}
\mathcal{M}^{a}_{N}\left(x_1, \ldots, x_N\right):=t^{\binom{a}{2}}\sum_{\substack{I \subset \{1, \ldots, N\} \\ |I|=a}} \prod_{i \in I}\prod_{j \in \{1, \ldots, N\}\setminus I} \frac{t x_i-x_j}{x_i-x_j} \Gamma_I.
\end{equation}

\end{defn}
\begin{remark}
	In \cite{DiFrancescoKedem2019}, there is also a generalization of the definition above, called the \emph{generalized Macdonald operators}, which associates a polynomial $P$ (and its symmetrization) with the coefficients of the Macdonald operator $\mathcal M^a_N$. We will only consider the case $P=1$, which is the definition of the original Macdonald operator above.
\end{remark}

The Macdonald operators act on $\displaystyle\mathcal F_N:=\C_{q,t}(x_1, \ldots, x_N)^{\mathfrak S_N}$, the space of symmetric rational functions in $N$ variables, and they commute with each other. Let us first prove the following theorem, which will later motivate Theorem \ref{conj: rank recursion}.

\begin{lemma}
	\label{lem: three-term rank recursion}
	The original Macdonald operators satisfy the following three-term rank recursion: 
	\begin{equation}
	\label{eq:three-term rank recursion}
	\begin{aligned}
		\M^{1}_{N+1}(x_1,\dots, x_{N+1})&=\frac{tx_N-x_{N+1}}{x_{N}-x_{N+1}}\M^{1}_{N}(x_1,\dots, x_{N})+\frac{tx_{N+1}-x_N}{x_{N+1}-x_N}\M^{1}_{N}(x_1,\dots, x_{N-1}, x_{N+1})\\
		&-t\M^{1}_{N-1}(x_1,\dots, x_{N-1})
	\end{aligned}
	\end{equation}
\end{lemma}
\begin{proof}
	We compare the coefficients of $\Gamma_i$ for any $i \in \{1, \dots, N\}$. Then, on the r.h.s. of \eqref{eq:three-term rank recursion}, the coefficients of $\Gamma_i$ are given by:
	\begin{equation*}
	\begin{aligned}
		&\bullet \textrm{For any fixed }i \in \{1, \dots, N-1\}:\\
		&\left(\frac{tx_{N+1}-x_N}{x_{N+1}-x_N}\frac{tx_i-x_{N+1}}{x_i-x_{N+1}}\prod_{j \in \{1, \dots, N-1\}\backslash \{i\}}\frac{tx_i-x_j}{x_i-x_j}\right)+\left(\frac{tx_N-x_{N+1}}{x_{N}-x_{N+1}}\frac{tx_i-x_{N}}{x_i-x_N}\prod_{j \in \{1, \dots, N-1\}\backslash \{i\}}\frac{tx_i-x_j}{x_i-x_j}\right)\\&
		-t\left(
		\prod_{j \in \{1,\dots,N-1\}\backslash \{i\}}\frac{tx_i-x_j}{x_i-x_j}\right)\\
		&=\left(\frac{tx_i-x_{N+1}}{x_i-x_{N+1}}\frac{tx_{N+1}-x_N}{x_{N+1}-x_N}+\frac{tx_N-x_{N+1}}{x_{N}-x_{N+1}}\frac{tx_i-x_N}{x_i-x_N}-t\right)\prod_{j \in \{1, \dots, N-1\}\backslash \{i\}}\frac{tx_i-x_j}{x_i-x_j}\\
		&=\frac{(t x_i-x_N) (t x_i-x_{N+1})}{(x_i-x_N)(x_i-x_{N+1})}\prod_{j \in \{1, \dots, N\}\backslash \{i\}}\frac{tx_i-x_j}{x_i-x_j}
		\\
		&=\prod_{j \in \{1, \dots, N+1\}\backslash \{i\}}\frac{tx_i-x_j}{x_i-x_j}\\
		&\bullet \textrm{For }i=N:  \frac{tx_N-x_{N+1}}{x_{N}-x_{N+1}}\prod_{j \in \{1, \dots, N-1\}}\frac{tx_N-x_j}{x_N-x_j}=\prod_{j \in \{1, \dots, N+1\}\backslash \{N\}}\frac{tx_N-x_j}{x_N-x_j}\\
		&\bullet \textrm{For }i=N+1: \frac{tx_{N+1}-x_N}{x_{N+1}-x_N}\prod_{j \in \{1, \dots, N-1\}}\frac{tx_{N+1}-x_j}{x_{N+1}-x_j}=\prod_{j \in \{1, \dots, N+1\}\backslash \{N+1\}}\frac{tx_{N+1}-x_j}{x_{N+1}-x_j}
	\end{aligned}
\end{equation*}
So we have all the coefficients of $\Gamma_i$ for $i \in \{1, \dots, N+1\}$ on the r.h.s. of \eqref{eq:three-term rank recursion} are the same as the coefficients of $\Gamma_i$ for $i \in \{1, \dots, N+1\}$ on the l.h.s. of \eqref{eq:three-term rank recursion}.
\end{proof}

\begin{remark}
	There are a couple of features in the previous proof that can be generalized:
	\begin{enumerate}
		\item Note that in the previous proof, the two special cases are $i=N$ and $i=N+1$, where the coefficient of $\Gamma_i$ on the r.h.s. of \eqref{eq:three-term rank recursion} only comes from the first two terms. This is also true if we pick any subset of two elements from $\{1, \dots, N+1\}$, and the coefficient of $\Gamma_i$ for $i$ in this subset on the r.h.s. of \eqref{eq:three-term rank recursion} only comes from the first two terms after relabeling the set of variables being acted on.
		\item For $a>1$, we have Lemma \ref{lem: rank recursion maximal case rank a>1}.
	\end{enumerate}
\end{remark}
Before showing the proof of Lemma \ref{lem: rank recursion maximal case rank a>1}, we first state some propositions that can be checked immediately from the definition of the Macdonald operators, along with some useful notations that we will use in the proof of Lemma \ref{lem: rank recursion maximal case rank a>1}.

Fix the following notation: 
\begin{itemize}
	\item $[\Gamma_K]\mathcal M_{N}^a(x_1, \ldots, x_N)$ denotes the coefficients of $\Gamma_K=\prod_{k \in K} \Gamma_k$ in the expansion of the Macdonald operator $\mathcal M_{N}^a(x_1, \ldots, x_N)$. 
	\item $[a,b]=\{a, a+1, \dots, b\}$ for $a \leq b$ 
\end{itemize}

\begin{prop}
	\label{prop:coefficients extracting}
	For $N>M$, consider the two Macdonald operators
	\[\mathcal M_{N}^a(x_{1},\ldots,x_{M},\ldots, x_{N}), \quad \mathcal M_{M+m}^a(x_{1}, \ldots, x_{M}, x_{i_1}, x_{i_2}\dots, x_{i_m})\]
	then for a set $K \subset[M] \sqcup \{i_1, \ldots, i_m\}$ where $\{i_1, \ldots, i_m\} \subset [M+1, N]$, we have:
	\begin{equation}
	\label{eq:coefficients extracting}
	\begin{aligned}
	&\qquad[\Gamma_K]\mathcal M_{N}^a(x_{1},\ldots,x_{M},\ldots, x_{N})\\
	&=\left([\Gamma_K]\mathcal M_{M+m}^a(x_{1}, \ldots, x_{M}, x_{i_1}, \dots, x_{i_m})\right)\prod_{\substack{k \in K\\ \ell \in [M+1, N]\backslash\{i_1, \ldots, i_m\}}}\frac{t x_k-x_\ell}{x_k-x_\ell}
	\end{aligned}
	\end{equation}
\end{prop}
The proof for the lemma below relies on some technical properties of the coefficients of the Macdonald operators which  we show in the Appendix \ref{appendix1}. 
\begin{lemma}
	\label{lem: rank recursion maximal case rank a>1}
	For $a>1$, 
	\begin{equation}	
		\label{eq: rank recursion maximal case rank a>1}
	\mathcal M_{N}^a(x_1, \ldots, x_N)=\sum_{\substack{I \subseteq [N-a,N] \\ I \neq \emptyset}}(-1)^{|I|-1}t^{\binom{|I|}{2}}\prod_{i \in I}\prod_{ j\in [N-a,N]\backslash I}\frac{tx_j-x_i}{x_j-x_i}\mathcal M_{N-|I|}^a(\{x_n\}_{n \in [N]\backslash I})
	\end{equation}
\end{lemma}

\begin{proof}
	We would like to show that for some subset $K \subset [N]$ with $|K|=a$, the coefficient of $\Gamma_K$ on the r.h.s. of \eqref{eq: rank recursion maximal case rank a>1} is the same as the coefficient of $\Gamma_K$ on the l.h.s. of \eqref{eq: rank recursion maximal case rank a>1}. Assuming that $K \cap [N-a,N]=\{i_1, \dots, i_m\}$ (which can possibly be empty) and apply Proposition \ref{prop:coefficients extracting} to the l.h.s., we need to show:
\begin{equation}
	\label{eq: K intersect tail}
	\begin{aligned}
	&\qquad\left([\Gamma_K]\mathcal M^a_{N-a+m-1}\left(x_1, \dots, x_M, x_{i_1}, \dots, x_{i_m}\right)\right)\prod_{k \in K}\prod_{\ell \in [N-a,N]\backslash \{i_1, \dots, i_m\}}\frac{t x_k-x_\ell}{x_k-x_\ell}\\
&=\sum_{\substack{I \subseteq [N-a,N], \\ I \neq \emptyset}}(-1)^{|I|-1}t^{\binom{|I|}{2}}\prod_{i \in I}\prod_{ j\in [N-a, N]\backslash I}\frac{tx_j-x_i}{x_j-x_i}[\Gamma_K]\mathcal M^a_{N-|I|}(\{x_n\}_{n \in [N]\backslash I})
	\end{aligned}
\end{equation}
Note that on the r.h.s., the summands that contain $\Gamma_K$ in $\mathcal M^a_{N-|I|}(\{x_n\}_{n \in [N]\backslash I})$ are those indexed by $I \subseteq [N-a,N]\backslash \{i_1, \dots, i_m\}$ (otherwise $\Gamma_K$ does not even exist on the r.h.s.). Then, we have: 
\begin{equation}
	\begin{aligned}
		&\qquad [\Gamma_K]\mathcal M^a_{N-|I|}(\{x_n\}_{n \in [N]\backslash I})\\
		&=[\Gamma_K]\mathcal M^a_{N-a+m-1}(\{x_1, \dots, x_{N-a-1}, x_{i_1}, \dots, x_{i_m}\})\prod_{\substack{k \in K\\ \ell \in [N-a,N]\backslash (I \sqcup \{i_1, \dots, i_m\})}}\frac{t x_k-x_\ell}{x_k-x_\ell}
	\end{aligned}
\end{equation}
Define $\widetilde I=I \sqcup \{i_1, \dots, i_m\}$, once again, it boils down to show the following identity:
\begin{equation}
	\label{eq: K intersect tail last factor}
\begin{aligned}
	&\qquad\left(\sum_{\substack{I \subseteq [N-a,N]\backslash \{i_1,\dots,i_m\} \\ I \neq \emptyset}}(-1)^{|I|-1}t^{\binom{|I|}{2}}\prod_{\substack{i \in I,k\in K, \\ \alpha\neq i,k,\\ \alpha \in [N-a,N]\backslash \widetilde{I}}}\frac{tx_\alpha-x_i}{x_\alpha-x_i}\frac{tx_k-x_\alpha}{x_k-x_\alpha}\right)+(-1)^{(a-1)}t^{\binom{a}{2}}\\
	&=\prod_{\substack{k \in K\\ \beta \in [N-a,N]\backslash \{i_1, \dots, i_m\}}}\frac{t x_k-x_\beta}{x_k-x_\beta}
\end{aligned}
\end{equation}
Define $S:=[N-a,N]\backslash \{i_1, \dots, i_m\}$ and apply Lemma \ref{lemma:appendix lemma 5} to the disjoint sets $(S,K)$ to finish the proof of the relation.
\end{proof}

Note that for any subset $K = \{k_0, \dots, k_{a}\} \subset \{1,\ldots,  N\}$ with $k_0<k_1<\dots<k_{a}$, by making the change of variables  $N-a+i\mapsto k_i$ for $i=0, \dots, a$, the following theorem is a direct generalization of Lemma \ref{lem: rank recursion maximal case rank a>1}.

\begin{thm}
	\label{conj: rank recursion}
	For any non-empty subset $K = \{k_0, \dots, k_{a}\} \subset \{1,\ldots,  N\}$ with $k_0<k_1<\dots<k_{a}$, we have:
	\begin{equation}
		\label{eq: rank recursion general}
		\mathcal M^a_{N}(x_1,\dots, x_N)=\sum_{\substack{J \subseteq K \\ J \neq \emptyset}}(-1)^{|J|-1}t^{\binom{|J|}{2}}\left(\prod_{j \in K \backslash J}\prod_{i \in J}\frac{tx_j-x_i}{x_j-x_i}\right)\mathcal M^a_{N-|J|}(\{x_k, k \in [1, N]\backslash J\})
	\end{equation}
\end{thm}

\begin{remark}
	Note that the coefficients $\displaystyle t^{\binom{|J|}{2}}\prod_{j \in K \backslash J}\prod_{i \in J}\frac{tx_j-x_i}{x_j-x_i}$ in Theorem \ref{conj: rank recursion} admits a determinantal formula. Specifically, for any non-empty subset $J \subseteq \{1,\ldots,m\}$, we have the following identity:
$$\det(N_J) = t^{\binom{|J|}{2}} \prod_{i \in J}\,\prod_{j \notin J} f_{ij}.$$
where $\displaystyle f_{ij} = \frac{t x_j - x_i}{x_j - x_i} \quad (i \neq j)$ and $N_J$ is the $|J| \times |J|$ matrix defined by
$$
N_{ii} = \prod_{k \neq i} f_{ik}, \qquad
N_{ij} = (f_{ij}-1)\prod_{k \notin \{i,j\}} f_{ik} \quad (i \neq j).
$$
and thus Eq. \eqref{eq: rank recursion general} can be written as 
\begin{equation}
	\mathcal M^a_{N}(x_1,\dots, x_N)=\sum_{\substack{J \subseteq K \\ J \neq \emptyset}}(-1)^{|J|-1}\det(N_J)\mathcal M^a_{N-|J|}(\{x_k, k \in [1, N]\backslash J\})
\end{equation}
\end{remark}

Similar to the $q$-Whittaker case, the Macdonald operators also have a formula for the $k$-th power. From here on, we denote $\bx=(x_1, \dots, x_N)$. We start with the following proposition.
\begin{prop}
	\label{prop: q-shift identity}
	Let $J,K \subseteq [N]$ and let $c(\bx), d(\bx)$ be two scalar rational functions in $x_1, \dots, x_N$. Define the operators
\[
C_J:=c\Gamma_J
\qquad
D_K:=d\Gamma_K.
\]
Then, 
$$D \circ C = \bigl(d(\bx) \cdot c(\bx)(q^K \bx)\bigr) \cdot \Gamma_{J+K}.$$
where $q^J \bx=(x_1, \dots,qx_j, \dots, x_N)$ is the vector obtained by multiplying $x_j$ by $q$ for each $j \in J$.
\end{prop}
\begin{proof}
	Take a rational function $f(\bx)$. By definition, we have
\[(C_Jf)(\bx)=c(\bx)f(q^J \bx)\]

Then, 
\begin{align*}
(D_K\circ C_J)(f)(\bx)&=D_K(C_Jf)(\bx)=d(\bx)(C_Jf)(q^\bx)=d(\bx)c(q^\bx)f(q^{J+K}\bx)\\
&=\bigl(d(\bx) \cdot c(q^K \bx)\bigr) \cdot \Gamma_{J+K}(f(\bx)).
\end{align*}
\end{proof}
One can generalize the previous proposition for a set of subsets $\bJ=\{J_1, \dots, J_m\}$ by repeating the same argument. Specifically, 
\begin{align*}
	(d\Gamma_K)\circ (c\Gamma_\bJ)(f)(\bx)=d(\bx)c(q^K \bx)\Gamma_K\Gamma_\bJ f(\bx).
\end{align*}
Let $\displaystyle A_i(\bx)=\prod_{j\neq i}\frac{tx_i-x_j}{x_i-x_j}$, then we have $\displaystyle \mathcal M^1_N=\sum_{i=1}^N A_i(\bx)\Gamma_i$. We are now ready to derive the formula for the $k$-th power of the Macdonald operator $\mathcal M^a_N$.

\begin{lemma}
	\label{lem: 1st Macdonald power formula}
	For any $k \in \mathbb N$, let $m_{<l}(j):=|\{i_r=j\mid i_r\in \{i_1, \dots, i_{l-1}\}\}|$, i.e., the number of times $j$ appears in $\{i_1, \dots, i_{l-1}\}$.
	\begin{equation}
		\label{eq: Macdonald power formula}
		(\mathcal M^1_N)^k=\sum_{i_1, \dots,i_k=1}^N\left(\prod_{l=1}^k\prod_{j \neq i_l}\frac{tq^{\mcard{l}{i_l}}x_{i_l}-q^{\mcard{l}{j}}x_j}{q^{\mcard{l}{i_l}}x_{i_l}-q^{\mcard{l}{j}}x_j}\right)\Gamma_{i_1}\dots \Gamma_{i_k}
	\end{equation}
\end{lemma}

\begin{proof}
We proceed by induction on $k$. The case $k=1$ reduces to the definition of the Macdonald operator, so assume that the following is true for $k-1$:
	\begin{equation}
		\label{eq: induction hypothesis}
		(\mathcal M^1_N)^{k-1}=\sum_{i_2, \dots,i_{k}=1}^N\left(\prod_{l=1}^{k-1}\prod_{j \neq i_l}\frac{tq^{\mcard{l}{i_l}}x_{i_l}-q^{\mcard{l}{j}}x_j}{q^{\mcard{l}{i_l}}x_{i_l}-q^{\mcard{l}{j}}x_j}\right)\Gamma_{i_1}\dots \Gamma_{i_{k-1}}\\
		=\sum_{i_1, \dots, i_k}^{N}B_{i_1, \dots, i_k}(\bx)\Gamma_{i_1}\dots \Gamma_{i_{k-1}}
	\end{equation}
Let $(\cM^1_N)^k=\cM^1_N\circ (\cM^1_N)^{k-1}$ and we would like to show the coefficient in 
	\[ (\cM^1_N)^k=\sum_{i_k=1}^NA_{i_k}(\bx)\Gamma_{i_k}\circ \left(\sum_{i_1, \dots, i_{k-1}}^{N}B_{i_1, \dots, i_{k-1}}(\bx)\Gamma_{i_1}\dots \Gamma_{i_{k-1}}\right)\]
matches the coefficients in the expansion of $(\cM^1_N)^k$ given by the \eqref{eq: Macdonald power formula}. By Proposition \ref{prop: q-shift identity}, we have:
	\begin{align*}
		&\quad A_{i_k}(\bx)\Gamma_{i_k}\left((B_{i_1, \dots, i_{k-1}}(\bx)\Gamma_{i_1}\dots\Gamma_{i_{k-1}}f)\right)(\bx)\\
		&=A_{i_k}(\bx)(\Gamma_{i_k}B_{i_1, \dots, i_{k-1}}(\bx))f(\bx^{(k)})
	\end{align*}
where $\bx^{(k)}$ is the vector obtained by multiplying $x_{i_l}$ by $q$ for each $l=1, \dots, k$. Considering the explicit expression \eqref{eq: induction hypothesis}, we have: 
	\begin{align*}
	\Gamma_{i_k}\left(\prod_{l=1}^{k-1}\prod_{j \neq i_l}\frac{tq^{\mcard{l}{i_l}}x_{i_l}-q^{\mcard{l}{j}}x_j}{q^{\mcard{l}{i_l}}x_{i_l}-q^{\mcard{l}{j}}x_j}\right)=\prod_{l=1}^{k-1}\prod_{j \neq i_l}\frac{tq^{\mcard{l}{i_l}+\delta_{i_k,i_l}}x_{i_l}-q^{\mcard{l}{j}+\delta_{i_k,j}}x_j}{q^{\mcard{l}{i_l}}x_{i_l}-q^{\mcard{l}{j}+\delta_{i_k,j}}x_j}
	\end{align*}
where $\delta_{i,j}$ is the usual Kronecker delta function. Finally, we have:
	\begin{align*}
	\quad A_{i_k}(\bx)(\Gamma_{i_k}B_{i_1, \dots, i_{k-1}}(\bx))=\prod_{l=1}^k\prod_{j \neq i_l}\frac{tq^{\mcard{l}{i_l}}x_{i_l}-q^{\mcard{l}{j}}x_j}{q^{\mcard{l}{i_l}}x_{i_l}-q^{\mcard{l}{j}}x_j}
	\end{align*}

\end{proof}

\begin{thm}
\label{thm: Macdonald level a power formula}
Let $N\geq 1$ and $a\in \{1, \dots, N\}$. For any $k \in \mathbb N$, we have:
\begin{equation}
\label{eq: Macdonald level a power formula}
(\mathcal{M}^a_N)^k = \sum_{\substack{I_1,\ldots,I_k \subset [N] \\ |I_\ell|=a}} \left(\prod_{\ell=1}^k \prod_{i \in I_\ell}\prod_{j \notin I_\ell} \frac{t\,q^{\tmcard{\ell}{i}}x_i - q^{\tmcard{\ell}{j}}x_j}{q^{\tmcard{\ell}{i}}x_i - q^{\tmcard{\ell}{j}}x_j}\right) \Gamma_{I_1}\cdots\Gamma_{I_k}. 
\end{equation}
\end{thm}

\begin{proof}
	We proceed similarly to the previous lemma by induction on $k$. Define $\displaystyle C_{I}(\bx):=\prod_{i \in I}\prod_{j \notin I} \frac{tx_i-x_j}{x_i-x_j}$, then we have $\displaystyle \mathcal M^a_N=\sum_{|I|=a}C_I(\bx)\Gamma_I$. Assuming the following is true for $k-1$:
\begin{equation}
\label{eq: induction hypothesis level a}	
\begin{aligned}
(\mathcal{M}^a_N)^{k-1} &= \sum_{\substack{I_1,\ldots,I_{k-1} \subset [N] \\ |I_\ell|=a}} \left(\prod_{\ell=1}^{k-1} \prod_{i \in I_\ell}\prod_{j \notin I_\ell} \frac{t\,q^{\tmcard{\ell}{i}}x_i - q^{\tmcard{\ell}{j}}x_j}{q^{\tmcard{\ell}{i}}x_i - q^{\tmcard{\ell}{j}}x_j}\right) \Gamma_{I_1}\cdots\Gamma_{I_{k-1}}\\
&=\sum_{\substack{I_1,\ldots,I_{k-1} \subset [N] \\ |I_\ell|=a}} B_{I_1, \dots, I_{k-1}}(\bx)\Gamma_{I_1}\cdots\Gamma_{I_{k-1}}.
\end{aligned}
\end{equation}
Then, we have: 
\begin{align*}
	\cM^a_N\circ (\mathcal{M}^a_N)^{k-1} &= \left(\sum_{|I_k|=a}C_{I_k}(\bx)\Gamma_{I_k}\right)\circ \left(\sum_{\substack{I_1,\ldots,I_{k-1} \subset [N] \\ |I_\ell|=a}} B_{I_1, \dots, I_{k-1}}(\bx)\Gamma_{I_1}\cdots\Gamma_{I_{k-1}}\right).
\end{align*}
Applying Proposition \ref{prop: q-shift identity}, we have:
\begin{align*}
&\quad C_{I_k}(\bx)\Gamma_{I_k}\left(B_{I_1, \dots, I_{k-1}}(\bx)\Gamma_{I_1}\cdots\Gamma_{I_{k-1}}f\right)(\bx)\\
&=C_{I_k}(\bx)(\Gamma_{I_k}B_{I_1, \dots, I_{k-1}}(\bx))f(\bx^{(k)})
\end{align*}
where $\bx^{(k)}$ is the vector obtained by multiplying $x_i$ by $q$ for each $i \in I_\ell$ and $\ell=1, \dots, k-1$. Considering the explicit expression of $B_{I_1, \dots, I_{k-1}}(\bx)$, we once again have: 
\[\begin{aligned}
	\Gamma_{I_k}B_{I_1, \dots, I_{k-1}}(\bx)&=\Gamma_{I_k}\left(\prod_{\ell=1}^{k-1} \prod_{i \in I_\ell}\prod_{j \notin I_\ell} \frac{t\,q^{\tmcard{\ell}{i}}x_i - q^{\tmcard{\ell}{j}}x_j}{q^{\tmcard{\ell}{i}}x_i - q^{\tmcard{\ell}{j}}x_j}\right)\\
	&=\prod_{\ell=1}^{k-1} \prod_{i \in I_\ell}\prod_{j \notin I_\ell} \frac{t\,q^{\tmcard{\ell}{i}+\delta_{i \in I_k}}x_i - q^{\tmcard{\ell}{j}+\delta_{j\in I_k}}x_j}{q^{\tmcard{\ell}{i}+\delta_{i \in I_k}}x_i - q^{\tmcard{\ell}{j}+\delta_{j\in I_k}}x_j}
\end{aligned}\]
and thus,
\begin{align*}
	C_I(\bx)(\Gamma_IB_{I_1, \dots, I_{k-1}}(\bx))f(\bx) = \prod_{\ell=1}^k \prod_{i \in I_\ell}\prod_{j \notin I_\ell} \frac{t\,q^{\tmcard{\ell}{i}}x_i - q^{\tmcard{\ell}{j}}x_j}{q^{\tmcard{\ell}{i}}x_i - q^{\tmcard{\ell}{j}}x_j}
\end{align*}
which finishes the proof.
\end{proof}

\section{Commutator relations for Macdonald operators power}
The power formula in the previous section allows us to derive some identities among the commutators of $\mathcal M^a_N$ and $\displaystyle \sum_{i=1}^N x_i$. Namely, we have the following result: 

\begin{thm}
Let $N \ge 1$, $a \in \{1,\ldots,N\}$, and let $q, t$ be parameters. The Macdonald operator $\mathcal{M}^a_N$ is

$$\mathcal{M}^a_N(x_1,\ldots,x_N) = \sum_{\substack{I \subset \{1,\ldots,N\} \\ |I|=a}} C_I(x)\,\Gamma_I,
\qquad
C_I(x) = \prod_{i \in I}\prod_{j \notin I} \frac{t x_i - x_j}{x_i - x_j},$$
Let $X = \sum_{i=1}^N x_i$ be the operator of multiplication by the sum of variables. For any integer $k \ge 0$, the commutator of the $k$-th power of $\mathcal{M}^a_N$ with $X$ is
\begin{equation}\label{eq: commutator identity}
[(\mathcal{M}^a_N)^k, X] = \sum_{j=0}^{k-1} (\mathcal{M}^a_N)^j \; R_1 \; (\mathcal{M}^a_N)^{k-1-j}
\end{equation}
where $R_1$ is the commutator of $\mathcal{M}^a_N$ with $X$ given by
$$R_1 = [\mathcal{M}^a_N, X] = (q-1) \sum_{\substack{I \subset [N] \\ |I| = a}} C_I(x) \left(\sum_{i \in I} x_i\right) \Gamma_I.$$
\end{thm}
\begin{proof}
    For any subset $I$, the operator $\Gamma_I$ acts as a ring homomorphism on the ring of rational functions $\mathfrak R_{\C(q;t)}:=\C(q;t)(x_1, \dots, x_N)$. Thus, denote the automorphism on $\mathfrak{R}_{\C(q;t)}$ via multiplication by $g\in \compratio $ we have
$$(\Gamma_I \circ M_g)(f)(x) = \Gamma_I(g \cdot f)(x) = g(q^I x) \cdot f(q^I x) = (M_{g(q^I x)} \circ \Gamma_I)(f)(x),$$
where $(q^I x)_i = q x_i$ if $i \in I$. Thus $\Gamma_I \circ M_g = M_{g(q^I x)} \circ \Gamma_I$. Apply this with $\displaystyle g = X := \sum_{i=1}^N x_i$ and note that $\displaystyle M_{X(q^Ix)}=X+ (q-1)\sum_{i\in I}x_i$ we have:
$$\Gamma_I \circ X = \Bigl(X + (q-1)\sum_{i \in I} x_i\Bigr) \circ \Gamma_I,$$
and therefore
\begin{equation}
    \label{eq:commutator rank 1}
    [\Gamma_I, X] = \Gamma_I X - X \Gamma_I = (q-1)\Bigl(\sum_{i \in I} x_i\Bigr)\,\Gamma_I.
\end{equation}
Since $\mathcal{M}^a_N = \sum_{|I|=a} C_I(x)\,\Gamma_I$ and each $C_I(x)$ is a multiplication operator (hence commutes with $X$),
\begin{equation}
    \label{eq:commutator rank a}
    R_1 = \sum_{|I|=a} C_I(x)\,[\Gamma_I, X] = (q-1) \sum_{|I|=a} C_I(x) \Bigl(\sum_{i \in I} x_i\Bigr)\,\Gamma_I. \tag{2.2}
\end{equation}
The family of difference operators with rational coefficients forms an associative algebra, thus, using the commutator identity $[AB, C] = A[B, C] + [A, C]B$ with $A = \mathcal{M}^a_N$, $B = (\mathcal{M}^a_N)^{k-1}$, $C = X$,
$$R_k = \mathcal{M}^a_N\,R_{k-1} + R_1\,(\mathcal{M}^a_N)^{k-1}.$$
This is a linear recurrence in $k$ and iterating $k-1$ times yields:
\begin{align*}
R_k &= \mathcal{M}^a_N R_{k-1} + R_1 (\mathcal{M}^a_N)^{k-1} \\
    &= \mathcal{M}^a_N\bigl(\mathcal{M}^a_N R_{k-2} + R_1 (\mathcal{M}^a_N)^{k-2}\bigr) + R_1 (\mathcal{M}^a_N)^{k-1} \\
    &= (\mathcal{M}^a_N)^2 R_{k-2} + \mathcal{M}^a_N R_1 (\mathcal{M}^a_N)^{k-2} + R_1 (\mathcal{M}^a_N)^{k-1} \\
    &= \cdots = \sum_{j=0}^{k-1} (\mathcal{M}^a_N)^j\,R_1\,(\mathcal{M}^a_N)^{k-1-j}.
\end{align*}
\end{proof}

Equivalently, expanding each power via \eqref{eq: Macdonald level a power formula} yields:
$$[(\mathcal{M}^a_N)^k, X] = \sum_{\substack{I_1,\ldots,I_k \subset [N] \\ |I_\ell| = a}} C_{I_1,\ldots,I_k}(x) \left(\sum_{i=1}^N (q^{\mu_i} - 1) x_i\right) \Gamma_{I_1} \cdots \Gamma_{I_k}$$

where $\mu_i = |\{\ell : i \in I_\ell\}|$ and $C_{I_1,\ldots,I_k}(x)$ is the coefficient from \eqref{eq: Macdonald level a power formula}:

$$C_{I_1,\ldots,I_k}(x) = \prod_{\ell=1}^k \prod_{i \in I_\ell}\prod_{j \notin I_\ell} \frac{t\,q^{m_{<\ell}(i)}x_i - q^{m_{<\ell}(j)}x_j}{q^{m_{<\ell}(i)}x_i - q^{m_{<\ell}(j)}x_j}, \qquad
m_{<\ell}(j) = |\{r < \ell : j \in I_r\}|.$$
This relation generalizes those in Section 6 of \cite{BC15} and will be useful in studying Macdonald processes and interacting particle systems.

\newpage
\appendix

\section{Identities Among $q$-shift  Coefficients}
\label{appendix1}
\begin{lemma}
    \label{lem: simple sum over all subsets [1,N]}
    Fix an integer $N\ge 1$ and let $x_1,\dots,x_N$ be pairwise distinct formal variables.  For each nonempty subset $I\subseteq [1,N]:=\{1,2,\dots,N\}$, we have the following equality
\begin{equation}
    \label{eq: simple sum over all nonempty subsets [1,N]}
    \sum_{\emptyset \neq I\subseteq [1,N]} (-1)^{|I|-1}t^{\binom{|I|}{2}}\prod_{i\in I}\ \prod_{j\in [1,N]\setminus I}\frac{tx_j-x_i}{x_j-x_i}=1.
\end{equation}

Equivalently, for $I$ possibly empty set, we have: 
\begin{equation}
    \label{eq: simple sum over all subsets [1,N]}
    \sum_{I\subseteq [1,N]} (-1)^{|I|}t^{\binom{|I|}{2}}
\prod_{i\in I}\prod_{j\in [1,N]\setminus I}
\frac{tx_j-x_i}{x_j-x_i}=0.
\end{equation}

\end{lemma}

\begin{proof}

If $N=1$ the sum has two terms ($I=\varnothing$ and $I=\{1\}$) that cancel to $0$, so the identity holds trivially.
Assume $N\ge 2$ and write $z=x_N$. Set $S=[1,N-1]=\{1,\dots,N-1\}$. Split the sum according to whether $N \in I$:
$$
F(z):=\sum_{I\subseteq [1,N]} (-1)^{|I|}t^{\binom{|I|}{2}}
\prod_{i\in I}\prod_{j\in [1,N]\setminus I}
\frac{tx_j-x_i}{x_j-x_i}
= A(z)+B(z),
$$
where
$$
\begin{aligned}
A(z)&=\sum_{J\subseteq S} (-1)^{|J|}t^{\binom{|J|}{2}}
P_J\;\prod_{i\in J}\frac{tz-x_i}{z-x_i},\quad B(z)=\sum_{J\subseteq S} (-1)^{|J|+1}t^{\binom{|J|+1}{2}}
P_J\;\prod_{j\in S\setminus J}\frac{tx_j-z}{x_j-z},\\
P_J&=\prod_{i\in J}\prod_{j\in S\setminus J}\frac{tx_j-x_i}{x_j-x_i}
\qquad (J\subseteq S).
\end{aligned}
$$
The expression $F(z)$ is a rational function in $z$. We prove $F(z)\equiv 0$ by showing it has no finite poles and vanishes at infinity.
Note that for each fraction involving $z$,
$$
\frac{tz-x_i}{z-x_i}\longrightarrow t,\qquad
\frac{tx_j-z}{x_j-z}\longrightarrow 1\qquad(z\to\infty).
$$
Hence
$$
\lim_{z\to\infty}F(z)
=\sum_{J\subseteq S}(-1)^{|J|}t^{\binom{|J|}{2}}t^{|J|}P_J
+\sum_{J\subseteq S}(-1)^{|J|+1}t^{\binom{|J|+1}{2}}P_J.
$$
Since $\binom{|J|+1}{2}=\binom{|J|}{2}+|J|$, we have $t^{\binom{|J|+1}{2}}=t^{\binom{|J|}{2}}t^{|J|}$, and the two sums cancel termwise and we obtain that 
$\displaystyle
\lim_{z\to\infty}F(z)=0.
$
The only possible poles of $F(z)$ lie at $z=x_k$ for $k\in S$.

Fix $k\in S$. In $A(z)$, a term with $k \in J$ contains the factor $\displaystyle \frac{tz-x_k}{z-x_k}$, which has a simple pole at $z=x_k$. Similarly, in $B(z)$, a term with $k \notin J $ contains the factor $\displaystyle \frac{tx_k-z}{x_k-z}$, which likewise has a simple pole at $z=x_k$. All other factors are regular at $z=x_k$.

For a given $J\subseteq S$ with $k\in J$, set $J'=J\setminus\{k\}$. Near $z=x_k$ we have the Laurent expansions
$$
\frac{tz-x_k}{z-x_k}=t+\frac{(t-1)x_k}{z-x_k},\qquad
\frac{tx_k-z}{x_k-z}=1-\frac{(t-1)x_k}{z-x_k}.
$$
Therefore the residue of $A(z)$ at $z=x_k$ coming from $J$ is
$$
(t-1)x_k\;(-1)^{|J|}t^{\binom{|J|}{2}}P_J
\prod_{i\in J'}\frac{tx_k-x_i}{x_k-x_i}.
$$
The residue of $B(z)$ at $z=x_k$ coming from $J'$ (note $k\notin J'$) is
$$
-(t-1)x_k\;(-1)^{|J'|+1}t^{\binom{|J'|+1}{2}}P_{J'}
\prod_{j\in S\setminus J',j\neq k}\frac{tx_j-x_k}{x_j-x_k}.
$$

Now expand $P_J$ and $P_{J'}$:
$$
\begin{aligned}
P_J&=\prod_{i\in J'}\prod_{j\in S\setminus J}\frac{tx_j-x_i}{x_j-x_i}
\;\cdot\;\prod_{j\in S\setminus J}\frac{tx_j-x_k}{x_j-x_k},\\[4pt]
P_{J'}&=\prod_{i\in J'}\prod_{j\in S\setminus J}\frac{tx_j-x_i}{x_j-x_i}
\;\cdot\;\prod_{i\in J'}\frac{tx_k-x_i}{x_k-x_i}.
\end{aligned}
$$
Multiplying the first line by $\prod_{i\in J'}\frac{tx_k-x_i}{x_k-x_i}$ and the second by $\prod_{j\in S\setminus J}\frac{tx_j-x_k}{x_j-x_k}$ yields the same product, hence
$$
P_J\prod_{i\in J'}\frac{tx_k-x_i}{x_k-x_i}
=P_{J'}\prod_{j\in S\setminus J}\frac{tx_j-x_k}{x_j-x_k}.
$$
Since $S\setminus J = S\setminus(J'\cup\{k\}) = (S\setminus J')\setminus\{k\}$, the right‑hand side is exactly the product appearing in the residue of $B(z)$. Moreover $|J|=|J'|+1$, so $(-1)^{|J|}=(-1)^{|J'|+1}$ and $\binom{|J|}{2}=\binom{|J'|+1}{2}$. Consequently the two residues  cancel pairwise across the full sums $A(z)$ and $B(z)$.

Thus, $F(z)$ has no poles on the complex plane. $F(z)$ is a rational function with no finite poles, hence a polynomial. A polynomial that tends to $0$ as $z\to\infty$ must be the zero polynomial. Therefore $F(z)\equiv 0$, which is precisely the stated identity.

\end{proof}

\begin{lemma}
    \label{lemma: single external variable sum lemma}
Fix an integer $N \ge 2$. Let $t$ be a parameter considered as an element in the coefficient field and let $x_1,\dots,x_N,z$ be indeterminates (or complex numbers) such that $x_i\neq x_j$ for $i\neq j$ and $z\neq x_i$ for all $i$.  
For each nonempty subset $I\subseteq[1,N]$, define
\[
A(I):=(-1)^{|I|-1}t^{\binom{|I|}{2}}
\prod_{i\in I \neq \emptyset \atop j \in [1,N]\setminus I}\ \frac{tx_j-x_i}{x_j-x_i}, \quad B(I;z):=\prod_{j\in [1,N]\setminus I}\frac{tz-x_j}{z-x_j}.
\]
Then the following identity holds:
\[
\sum_{\emptyset\neq I\subseteq[1,N]} A(I)B(I;z)=\prod_{j=1}^N\frac{tz-x_j}{z-x_j}.
\]
\end{lemma}

\begin{proof}
Define the rational functions
\[
F(z):=\sum_{\emptyset\neq I\subseteq[1,N]} A(I)B(I;z),
\qquad
R(z):=\prod_{j=1}^N\frac{tz-x_j}{z-x_j},
\qquad
D(z):=F(z)-R(z).
\]
Using the same strategy as in the proof of Lemma \ref{lem: simple sum over all subsets [1,N]}, we will prove $D(z)\equiv 0$ by showing $D(z)$ has vanishing residue at $z=x_r$ ($r\in[1,N]$), which are all simple poles. Hence, $D$ is a polynomial in $z$ since $D(z)$ is bounded as $z\to\infty$. Fix $r\in[1,N]$. The only possible poles of $F(z)$ and $R(z)$ occur at $z=x_1,\dots,x_N$, and they are at worst simple poles. Since
\[
\operatorname*{Res}_{z=x_r}\left(\frac{tz-x_r}{z-x_r}\right)=(t-1)x_r.
\]
and that all factors with $j\neq r$ are regular at $z=x_r$, hence
\begin{align*}
\operatorname*{Res}_{z=x_r}R(z)
&=\operatorname*{Res}_{z=x_r}\left(\frac{t z-x_r}{z-x_r}\right)\cdot \prod_{\substack{j=1\\ j\neq r}}^n\left.\frac{t z-x_j}{z-x_j}\right|_{z=x_r}=(t-1)x_r\cdot \prod_{\substack{j=1\\ j\neq r}}^n \frac{t x_r-x_j}{x_r-x_j}.
\end{align*}
Similarly, we can compute the residue of $F(z)$ at $z=x_r$:

\begin{align*}
\operatorname*{Res}_{z=x_r}B(I;z)
&=\operatorname*{Res}_{z=x_r}\left(\frac{t z-x_r}{z-x_r}\right)\cdot 
\prod_{\substack{j\in[1,N]\setminus I\\ j\neq r}}\left.\frac{t z-x_j}{z-x_j}\right|_{z=x_r}=(t-1)x_r\cdot \prod_{\substack{j\in[1,N]\setminus I\\ j\neq r}}\frac{t x_r-x_j}{x_r-x_j}.
\end{align*}
Therefore,
\begin{equation}\label{eq:resF-1}
\operatorname*{Res}_{z=x_r}F(z)
=(t-1)x_r\sum_{\substack{\emptyset\neq I\subseteq[1,N]\\ r\notin I}}
A(I)\cdot \prod_{\substack{j\in[1,N]\setminus I\\ j\neq r}}\frac{t x_r-x_j}{x_r-x_j}.
\end{equation}

Now expand $A(I)$ by separating the $j=r$ factor inside the product over $j\in[1,N]\setminus I$:
\[
A(I)
=(-1)^{|I|-1}t^{\binom{|I|}{2}}
\left(\prod_{i\in I}\frac{t x_r-x_i}{x_r-x_i}\right)\cdot
\left(\prod_{i\in I}\prod_{j\in([1,N]\setminus I)\setminus\{r\}}\frac{t x_j-x_i}{x_j-x_i}\right).
\]
Combine with \eqref{eq:resF-1}, the product
\[
\left(\prod_{i\in I}\frac{t x_r-x_i}{x_r-x_i}\right)\cdot
\left(\prod_{j\in([1,N]\setminus I)\setminus \{r\}}\frac{t x_r-x_j}{x_r-x_j}\right)
=\prod_{\substack{s\in[1,N]\\ s\neq r}}\frac{t x_r-x_s}{x_r-x_s}
\]
is independent of $I$, hence
\begin{equation}  
\label{eq:resF-2}
\begin{aligned}
\operatorname*{Res}_{z=x_r}F(z)
&=(t-1)x_r\left(\prod_{\substack{s\in[1,N]\\ s\neq r}}\frac{t x_r-x_s}{x_r-x_s}\right)
\sum_{\substack{\emptyset\neq I\subseteq[1,N]\setminus\{r\}}}
(-1)^{|I|-1}t^{\binom{|I|}{2}}
\prod_{i\in I}\prod_{j\in([1,N]\setminus I)\setminus\{r\}}\frac{t x_j-x_i}{x_j-x_i}.
\end{aligned}
\end{equation}
Now rewrite the remaining sum as a sum over subsets of $[1,N]\setminus\{r\}$ and applying Lemma \ref{lem: simple sum over all subsets [1,N]} for the set of variables $\{x_s\}_{s\in[1,N]\setminus\{r\}}$, we have the only summand in \eqref{eq:resF-2} that is not in \eqref{eq: simple sum over all subsets [1,N]} is the one corresponds to the empty set, which equals $1$. Hence
\[
\operatorname*{Res}_{z=x_r}F(z)
=(t-1)x_r\cdot \left(\prod_{\substack{s\in[1,N]\\ s\neq r}}\frac{t x_r-x_s}{x_r-x_s}\right)
=\operatorname*{Res}_{z=x_r}R(z).
\]
So $\displaystyle \operatorname*{Res}_{z=x_r}D(z)=0$ for every $r\in[1,N]$. Since the only possible finite poles of $D(z)$ are among $\{x_1,\dots,x_N\}$, all these poles are removable, hence $D(z)$ has no poles and is a bounded polynomial in $z$ and so must be a constant function in $z$.

Note that $B(I;0)=1$ for every $I\subseteq[1,N]$, and also $R(0)=1$. This implies
\[
F(0)=\sum_{\emptyset\neq I\subseteq[n]} A(I)\cdot 1
=\sum_{\emptyset\neq I\subseteq[1,N]} (-1)^{|I|-1}t^{\binom{|I|}{2}}
\prod_{i\in I}\ \prod_{j\in [n]\setminus I}\frac{ty_j-y_i}{y_j-y_i}.
\]
Again, by Lemma \ref{lem: simple sum over all subsets [1,N]}, this sum equals $1$. Hence $D(0)=F(0)-R(0)=0$.

Since $D(z)$ is constant and $D(0)=0$, we conclude $D(z)\equiv 0$, i.e. $F(z)=R(z)$.
\end{proof}

\begin{lemma}
    \label{lemma: appendix Lemma A.4}
Fix integers $N\ge 1$ and $M\ge 1$ and consider two  sets of pairwise distinct variables $x_1,\dots,x_N$ and $z_1,\dots,z_M$, such that $x_i\neq x_j$ for $i\neq j$ and $z_\ell\neq x_j$ for all $\ell,j$. For each subset $I\subseteq[1,N]$, define
\[
\begin{aligned}
    &E(I):=\prod_{i\in I}\ \prod_{j\in [1,N]\setminus I}\frac{tx_j-x_i}{x_j-x_i}
\qquad(\text{with }E(\emptyset)=1),\\
&A(I):=(-1)^{|I|-1}t^{\binom{|I|}{2}}E(I),\\
&B_M(I;z_1,\dots,z_M):=\prod_{\ell=1}^M\ \prod_{j\in [1,N]\setminus I}\frac{tz_\ell-x_j}{z_\ell-x_j}.
\end{aligned}
\]
Then, the following identity holds:
\[
\sum_{\emptyset\neq I\subseteq[1,N]} A(I)B_M(I;z_1,\dots,z_M)
\;=\;
\prod_{\ell=1}^M\ \prod_{j=1}^N\frac{tz_\ell-x_j}{z_\ell-x_j}.
\]

\end{lemma}

\begin{proof}
We prove the identity by induction on $M\ge 1$, for all $N\ge 1$ simultaneously.

\paragraph{Base case $M=1$.}
When $M=1$, the statement becomes
\[
\sum_{\emptyset\neq I\subseteq[1,N]} (-1)^{|I|-1}t^{\binom{|I|}{2}}
\left(\prod_{i\in I}\prod_{j\in [1,N]\setminus I}\frac{tx_j-x_i}{x_j-x_i}\right)
\left(\prod_{j\in [1,N]\setminus I}\frac{tz_1-x_j}{z_1-x_j}\right)
=
\prod_{j=1}^N\frac{tz_1-x_j}{z_1-x_j},
\]
which is exactly Lemma \ref{lemma: single external variable sum lemma} with $z=z_1$. Assume the equality holds for $M-1\ge 1$ and for all $N\ge 1$. Fix $N\ge 1$ and prove it for $M$. Define
\[
L_M(z_M)\;:=\;\sum_{\emptyset\neq I\subseteq[1,N]} A(I)B_M(I;z_1,\dots,z_M),
\qquad
R_M(z_M)\;:=\;\prod_{\ell=1}^M\ \prod_{j=1}^N\frac{tz_\ell-x_j}{z_\ell-x_j}.
\]
viewed as rational functions in the single variable $z_M$ over the coefficient field $\displaystyle \mathbb{\C}(t,x_1,\dots,x_N,z_1,\dots,z_{M-1})$. For each fixed $j$, the factor $\displaystyle \frac{t z_M-x_j}{z_M-x_j}$ has a simple pole at $z_M=x_j$, and no other finite poles. Hence both $L_M(z_M)$ and $R_M(z_M)$ can only have (at worst) simple poles at $z_M=x_1,\dots,x_N$.

Set the difference
\[
D_M(z_M)\;:=\;L_M(z_M)-R_M(z_M)
\]
Once again, we force $D_M$ to be a polynomial by showing $D_M$ has no poles at $z_M=x_1,\dots,x_N$ and is bounded as $z_M\to\infty$. Fix $r\in[1,N]$. We compute $\displaystyle \operatorname*{Res}_{z_M=x_r}L_M(z_M)$ and $\displaystyle \operatorname*{Res}_{z_M=x_r}R_M(z_M)$. Write $R_M(z_M)$ as
\[
R_M(z_M)
=
\left(\prod_{\ell=1}^{M-1}\ \prod_{j=1}^N\frac{tz_\ell-x_j}{z_\ell-x_j}\right)
\cdot
\left(\prod_{j=1}^N\frac{tz_M-x_j}{z_M-x_j}\right).
\]
Only the $j=r$ factor in the second product contributes a pole at $z_M=x_r$. We get:
\begin{align}
\operatorname*{Res}_{z_M=x_r}R_M(z_M)
&=(t-1)x_r
\left(\prod_{\ell=1}^{M-1}\ \prod_{j=1}^N\frac{tz_\ell-x_j}{z_\ell-x_j}\right)
\left(\prod_{\substack{j=1\\ j\neq r}}^N\frac{tx_r-x_j}{x_r-x_j}\right).
\label{eq:ResR}
\end{align}
Similarly, for a fixed subset $I\subseteq[1,N]$, the factor depending on $z_M$ inside $B_M(I;z_1,\dots,z_M)$ is $\displaystyle \prod_{j\in [1,N]\setminus I}\frac{tz_M-x_j}{z_M-x_j}$, which has a pole at $z_M=x_r$ if and only if $r\in [1,N]\setminus I$, i.e. $r\notin I$. Hence
\[
\operatorname*{Res}_{z_M=x_r}L_M(z_M)
=
\sum_{\substack{\emptyset\neq I\subseteq[1,N]\\ r\notin I}}
A(I)\left(\prod_{\ell=1}^{M-1}\ \prod_{j\in [1,N]\setminus I}\frac{tz_\ell-x_j}{z_\ell-x_j}\right)
\operatorname*{Res}_{z_M=x_r}\left(\prod_{j\in [1,N]\setminus I}\frac{tz_M-x_j}{z_M-x_j}\right).
\]
For such $I$, write $[1,N]\setminus I=J'\cup\{r\}$ where $J':=[1,N]\setminus(I\cup\{r\}) \subseteq [1,N]\setminus\{r\}$.
Then
\[
\prod_{j\in [1,N]\setminus I}\frac{tz_M-x_j}{z_M-x_j}
=
\left(\frac{tz_M-x_r}{z_M-x_r}\right)\cdot
\left(\prod_{j\in J'}\frac{tz_M-x_j}{z_M-x_j}\right),
\]
so
\begin{align}
\operatorname*{Res}_{z_M=x_r}\left(\prod_{j\in [1,N]\setminus I}\frac{tz_M-x_j}{z_M-x_j}\right)
&=
(t-1)x_r\cdot
\prod_{j\in J'}\frac{tx_r-x_j}{x_r-x_j}.
\label{eq:ResB}
\end{align}
Also, for $\ell\le M-1$ we can factor
\[
\prod_{j\in [1,N]\setminus I}\frac{tz_\ell-x_j}{z_\ell-x_j}
=
\left(\frac{tz_\ell-x_r}{z_\ell-x_r}\right)\cdot
\prod_{j\in J'}\frac{tz_\ell-x_j}{z_\ell-x_j}.
\]
Hence
\begin{align}
\prod_{\ell=1}^{M-1}\ \prod_{j\in [1,N]\setminus I}\frac{tz_\ell-x_j}{z_\ell-x_j}
&=
\left(\prod_{\ell=1}^{M-1}\frac{tz_\ell-x_r}{z_\ell-x_r}\right)
\cdot
\left(\prod_{\ell=1}^{M-1}\ \prod_{j\in J'}\frac{tz_\ell-x_j}{z_\ell-x_j}\right).
\label{eq:Bm_factor}
\end{align}

Now expand $A(I)$. Since $r\notin I$, in $\displaystyle E(I)=\prod_{i\in I}\prod_{j\in [1,N]\setminus I}\frac{tx_j-x_i}{x_j-x_i}$ we may split the $j$-product over $[1,N]\setminus I=J'\cup\{r\}$:
\begin{align}
E(I)
&=\left(\prod_{i\in I}\frac{tx_r-x_i}{x_r-x_i}\right)\cdot
\left(\prod_{i\in I}\ \prod_{j\in J'}\frac{tx_j-x_i}{x_j-x_i}\right).
\label{eq:E_split}
\end{align}
Multiply \eqref{eq:ResB} with the factor $\displaystyle \prod_{i\in I}\frac{t x_r-x_i}{x_r-x_i}$ coming from \eqref{eq:E_split}. Noting that $I\sqcup J'=[1,N]\setminus\{r\}$, we obtain:
\begin{align}
\left(\prod_{i\in I}\frac{tx_r-x_i}{x_r-x_i}\right)\cdot
\left(\prod_{j\in J'}\frac{tx_r-x_j}{x_r-x_j}\right)
&=\prod_{\substack{p=1\\ p\neq r}}^N\frac{tx_r-x_p}{x_r-x_p},
\label{eq:key_constant}
\end{align}
which is independent of $I$. Combining \eqref{eq:ResB}, \eqref{eq:Bm_factor}, \eqref{eq:E_split}, and \eqref{eq:key_constant}, we get
\begin{align*}
\operatorname*{Res}_{z_M=x_r}L_M(z_M)
&=(t-1)x_r
\left(\prod_{\substack{p=1\\ p\neq r}}^N\frac{tx_r-x_p}{x_r-x_p}\right)
\left(\prod_{\ell=1}^{M-1}\frac{tz_\ell-x_r}{z_\ell-x_r}\right)
\\
&\quad\times
\sum_{\substack{\emptyset\neq I\subseteq[1,N]\\ r\notin I}}
(-1)^{|I|-1}t^{\binom{|I|}{2}}
\left(\prod_{i\in I}\ \prod_{j\in J'}\frac{tx_j-x_i}{x_j-x_i}\right)
\left(\prod_{\ell=1}^{M-1}\ \prod_{j\in J'}\frac{tz_\ell-x_j}{z_\ell-x_j}\right).
\end{align*}
Now reindex subsets $I\subseteq[1,N]\setminus\{r\}$ (still nonempty) and interpret $J'=([1,N]\setminus\{r\})\setminus I$. The inner sum is exactly the  l.h.s. for the $(N-1)$-tuple $(x_p)_{p\neq r}$ and $(z_1,\dots,z_{M-1})$. By the induction hypothesis (valid for all $N$, hence for $N-1$) it equals
\[
\prod_{\ell=1}^{M-1}\ \prod_{\substack{j=1\\ j\neq r}}^N\frac{tz_\ell-x_j}{z_\ell-x_j}.
\]
Therefore
\begin{align}
\operatorname*{Res}_{z_M=x_r}L_M(z_M)
&=(t-1)x_r
\left(\prod_{\substack{p=1\\ p\neq r}}^N\frac{tx_r-x_p}{x_r-x_p}\right)
\left(\prod_{\ell=1}^{M-1}\frac{tz_\ell-x_r}{z_\ell-x_r}\right)
\left(\prod_{\ell=1}^{M-1}\ \prod_{\substack{j=1\\ j\neq r}}^N\frac{tz_\ell-x_j}{z_\ell-x_j}\right)
\nonumber\\
&=(t-1)x_r
\left(\prod_{\substack{p=1\\ p\neq r}}^n\frac{tx_r-x_p}{x_r-x_p}\right)
\left(\prod_{\ell=1}^{M-1}\ \prod_{j=1}^N\frac{tz_\ell-x_j}{z_\ell-x_j}\right),
\label{eq:ResL}
\end{align}
which matches \eqref{eq:ResR}. Thus for every $r$, $L_M$ and $R_M$ have the same residue at $z_M=x_r$, hence $D_M$ has no pole at $z_M=x_r$ and is a polynomial in $z_M$.

Note that 
\[
\frac{tz_M-x_j}{z_M-x_j}\to t
\quad\text{as }z_M\to\infty,
\]
so $R_M(z_M)$ has a finite limit as $z_M\to\infty$. Similarly, for each fixed $I$, the product $\displaystyle \prod_{j\in [1,N]\setminus I}\frac{t z_M-x_j}{z_M-x_j}$ tends to $t^{|[1,N]\setminus I|}$, so each summand in $L_M(z_M)$ has a finite limit and thus $L_M(z_M)$  has a finite limit. Therefore $D_M(z_M)$ is a polynomial bounded at infinity and so must be constant in $z_M$.

Since,
\[
B_M(I;z_1,\dots,z_{M-1},0)=B_{M-1}(I;z_1,\dots,z_{M-1}).
\]
Hence
\[
L_M(0)=\sum_{\emptyset\neq I\subseteq[1,N]}A(I)B_{M-1}(I;z_1,\dots,z_{M-1}),
\qquad
R_M(0)=\prod_{\ell=1}^{M-1}\ \prod_{j=1}^N\frac{tz_\ell-y_j}{z_\ell-y_j}.
\]
By the induction hypothesis for $M-1$, $L_M(0)=R_M(0)$, so $D_M(0)=0$. Since $D_M$ is constant in $z_M$, it follows that $D_M(z_M)\equiv 0$, i.e. $L_M(z_M)=R_M(z_M)$. This completes the induction and proves the identity for all $M\ge 1$.
\end{proof}

\begin{lemma}
    \label{lemma:appendix lemma 5}
    Let $S$ and $K$ be finite disjoint index sets with
\[
|S|=n\ge 1,\qquad |K|=m\ge 0.
\]
Let $\{x_r\}_{r\in S\cup K}$ be variables such that $x_u\neq x_v$ whenever $u\neq v$ and $u,v\in S\cup K$. For each subset $\emptyset \neq I\subseteq S$, define
\[
E_S(I)\;:=\;\prod_{i\in I}\ \prod_{\alpha\in S\setminus I}\frac{t x_\alpha-x_i}{x_\alpha-x_i}
\quad(\text{with }E_S(\emptyset)=1),\quad
B_{S,K}(I)\;:=\;\prod_{k\in K}\ \prod_{\alpha\in S\setminus I}\frac{t x_k-x_\alpha}{x_k-x_\alpha}
\quad(\text{with }B_{S,K}(S)=1).
\]
Then the following identity holds:
\begin{equation}
\label{eq:proper_subset_completion_multi_external}
\left(\sum_{\substack{\emptyset\neq I\subsetneq S}}
(-1)^{|I|-1}t^{\binom{|I|}{2}}
E_S(I)B_{S,K}(I)\right)
\;+\;
(-1)^{n-1}t^{\binom{n}{2}}
\;=\;
\prod_{k\in K}\ \prod_{\beta\in S}\frac{t x_k-x_\beta}{x_k-x_\beta}.
\end{equation}

\end{lemma}

\begin{proof}

Fix any ordering of the sets:
\[
S=\{\beta_1,\dots,\beta_n\}\quad(\text{all distinct}),\qquad
K=\{k_1,\dots,k_m\}.
\]
Define
\[
y_j:=x_{\beta_j}\ \ (1\le j\le n),\qquad z_\ell:=x_{k_\ell}\ \ (1\le \ell\le m).
\]
For any subset $J\subseteq[1,n]$, write $J^c=[1,n]\setminus J$ and define 
\[
E(J):=\prod_{i\in J}\ \prod_{j\in [1,n]\setminus J}\frac{ty_j-y_i}{y_j-y_i},\quad 
A(J):=(-1)^{|J|-1}t^{\binom{|J|}{2}}E(J)\quad(J\neq\emptyset),
\]
\[
B_m(J;z_1,\dots,z_m)\;:=\;\prod_{\ell=1}^m\ \prod_{j\in J^c}\frac{tz_\ell-y_j}{z_\ell-y_j}.
\]
By Lemma \ref{lemma: appendix Lemma A.4}, we have
\begin{equation}
\label{eq:lemma4_applied}
\sum_{\emptyset\neq J\subseteq[1,n]} A(J)B_m(J;z_1,\dots,z_m)
=
\prod_{\ell=1}^m\ \prod_{j=1}^n \frac{t z_\ell-y_j}{z_\ell-y_j}.
\end{equation}

Given any $J\subseteq[1,n]$, define the corresponding subset of $S$ by
\[
I(J)\;:=\;\{\beta_j:\ j\in J\}\subseteq S, \quad
S\setminus I(J)=\{\beta_j:\ j\in [1,n]\setminus J\}.
\]

Compute:
\begin{align*}
E(J)&=\prod_{i\in J}\ \prod_{j\in [1,n]\setminus J}\frac{ty_j-y_i}{y_j-y_i}=\prod_{i\in J}\ \prod_{j\in [1,n]\setminus J}\frac{tx_{\beta_j}-x_{\beta_i}}{x_{\beta_j}-x_{\beta_i}}=\prod_{\beta_i\in I(J)}\ \prod_{\beta_j\in S\setminus I(J)}\frac{tx_{\beta_j}-x_{\beta_i}}{x_{\beta_j}-x_{\beta_i}}=E_S(I(J)).
\end{align*}

\begin{align*}
B_m(J;z_1,\dots,z_m)=\prod_{\ell=1}^m\ \prod_{j\in J^c}\frac{tz_\ell-y_j}{z_\ell-y_j}=\prod_{\ell=1}^m\ \prod_{j\in J^c}\frac{tx_{k_\ell}-x_{\beta_j}}{x_{k_\ell}-x_{\beta_j}}=\prod_{k\in K}\ \prod_{\beta\in S\setminus I(J)}\frac{tx_k-x_\beta}{x_k-x_\beta}=B_{S,K}(I(J)).
\end{align*}
Then, \eqref{eq:lemma4_applied} becomes
\begin{equation}
\label{eq:sum_all_nonempty_I}
\sum_{\emptyset\neq I\subseteq S}
(-1)^{|I|-1}t^{\binom{|I|}{2}}E_S(I)B_{S,K}(I)
\;=\;
\prod_{k\in K}\ \prod_{\beta\in S}\frac{t x_k-x_\beta}{x_k-x_\beta}.
\end{equation}

On the left-hand side of \eqref{eq:sum_all_nonempty_I}, separate the contribution of $I=S$:
\[
\sum_{\emptyset\neq I\subseteq S}(\cdots)
=
\left(\sum_{\substack{\emptyset\neq I\subsetneq S}}(\cdots)\right)
+
\left[(-1)^{|S|-1}t^{\binom{|S|}{2}}E_S(S)B_{S,K}(S)\right].
\]
Note that both products defining $E_S(S)$ and $B_{S,K}(S)$ are empty products, hence $E_S(S)=1 \ \textrm{and}\ B_{S,K}(S)=1$. Also $|S|=n$, so the summand corresponding to $I=S$ equals $(-1)^{n-1}t^{\binom{n}{2}}$ and we have \eqref{eq:sum_all_nonempty_I} is equivalent to
\[
\left(\sum_{\substack{\emptyset\neq I\subsetneq S}}
(-1)^{|I|-1}t^{\binom{|I|}{2}}E_S(I)B_{S,K}(I)\right)
\;+\;
(-1)^{n-1}t^{\binom{n}{2}}
\;=\;
\prod_{k\in K}\ \prod_{\beta\in S}\frac{t x_k-x_\beta}{x_k-x_\beta},
\]
\end{proof}

\section{\texorpdfstring{Determinant interpretation of the $q$-shift coefficients}{Determinant interpretation of the q-shift coefficients}}
\label{appendix2}
Let $m \ge 1$ and let $x_1,\ldots,x_m$ be pairwise distinct variables. Let $t$ be a parameter with $t \neq 1$. Define

$$f_{ij} = \frac{t x_j - x_i}{x_j - x_i} \qquad (i \neq j).$$

Define the $m \times m$ matrix $N$ by

$$
N_{ii} = \prod_{k \neq i} f_{ik}, \qquad
N_{ij} = (f_{ij}-1)\prod_{k \notin \{i,j\}} f_{ik} \quad (i \neq j).
$$

For any subset $J \subseteq \{1,\ldots,m\}$, let $N_J$ denote the principal submatrix of $N$ with rows and columns indexed by $J$.

\begin{thm}
\label{thm: determinant identity}
For any nonempty subset $J \subseteq \{1,\ldots,m\}$, we have the following identity:
$$\det(N_J) = t^{\binom{|J|}{2}} \prod_{i \in J}\,\prod_{j \notin J} f_{ij}.$$
In particular, $\displaystyle \det(N) = t^{\binom{m}{2}}$, and $\displaystyle \det(N_\varnothing) = 1$ by the usual convention.
\end{thm}

\begin{proof}
Let $D = \operatorname{diag}(d_1,\ldots,d_m)$ with $d_i = \prod_{k \neq i} f_{ik}$.  
Let $M$ be the $m \times m$ matrix
$$M_{ij} = \frac{(t-1)x_j}{t x_j - x_i} \qquad (i,j = 1,\ldots,m).$$
Then, via direct entry-wise computation, we have $N = D M$.  Write $M = (t-1)\,A\,X$ where
$$A_{ij} = \frac{1}{t x_j - x_i}, \qquad X = \operatorname{diag}(x_1,\ldots,x_m).$$
For a principal submatrix indexed by $J \subseteq \{1,\ldots,m\}$ with $|J| = r$, let the elements of $J$ be $j_1 < j_2 < \cdots < j_r$. Then
$$M_J = (t-1)^r \, A_J \, X_J,$$
with $\displaystyle (A_J)_{ab} = \frac{1}{t x_{j_b} - x_{j_a}}$ and $X_J = \operatorname{diag}(x_{j_1},\ldots,x_{j_r})$.

Observe $A_J = -B$ where $\displaystyle B_{ab} = \frac{1}{x_{j_a} - t x_{j_b}}$. Such a matrix $B$ is called the Cauchy matrix, see for example \cite{Schechter1959,cauchy1841exercices2} with determinant given by
$$\det(B) = \frac{\prod_{1 \le a < b \le r} (x_{j_a} - x_{j_b})(t x_{j_b} - t x_{j_a})}
{\prod_{a=1}^r \prod_{b=1}^r (x_{j_a} - t x_{j_b})}= t^{\binom{r}{2}} (-1)^{\binom{r}{2}}
\frac{\prod_{a<b} (x_{j_a} - x_{j_b})^2}
{\prod_{a,b} (x_{j_a} - t x_{j_b})}.$$
Then, since  $\det(A_J) = (-1)^r \det(B)$, we have: 
$$\det(A_J) = (-1)^r \det(B)
= t^{\binom{r}{2}} (-1)^{r + \binom{r}{2}}
\frac{\prod_{a<b} (x_{j_a} - x_{j_b})^2}
{\prod_{a,b} (x_{j_a} - t x_{j_b})}.$$
and thus
\begin{equation}
    \label{eq: MJ determinant}
    \det(M_J) = (t-1)^r \det(A_J) \det(X_J)
= (t-1)^r \, t^{\binom{r}{2}} (-1)^{r + \binom{r}{2}}
\frac{\prod_{a<b} (x_{j_a} - x_{j_b})^2}
{\prod_{a,b} (x_{j_a} - t x_{j_b})}
\;\prod_{b=1}^r x_{j_b}.
\end{equation}
Since $N_J = D_J M_J$,
$$\det(N_J) = \Bigl(\prod_{i \in J} d_i\Bigr) \det(M_J).$$

Now, $$\prod_{i \in J} d_i = \prod_{i \in J} \prod_{k \neq i} f_{ik}
= \Bigl(\prod_{\substack{i,k \in J \\ i \neq k}} f_{ik}\Bigr)
\Bigl(\prod_{i \in J} \prod_{k \notin J} f_{ik}\Bigr).$$
Thus
$$\det(N_J) = \Bigl(\prod_{\substack{i,k \in J \\ i \neq k}} f_{ik}\Bigr)
\Bigl(\prod_{i \in J} \prod_{k \notin J} f_{ik}\Bigr)
\det(M_J).$$
Applying Lemma \ref{lemma:submatrix determinant as t power } below, we have:
$$\det(N_J) = t^{\binom{r}{2}} \prod_{i \in J} \prod_{k \notin J} f_{ik},$$
as desired.

\end{proof}

\begin{lemma}
    \label{lemma:submatrix determinant as t power }
    $$\Bigl(\prod_{\substack{i,k \in J \\ i \neq k}} f_{ik}\Bigr) \det(M_J) = t^{\binom{r}{2}}.$$
\end{lemma}
\begin{proof}
 $$\prod_{\substack{i,k \in J \\ i \neq k}} f_{ik}
= \frac{\prod_{a \neq b} (t x_{j_b} - x_{j_a})}
{\prod_{a \neq b} (x_{j_b} - x_{j_a})}.$$

The numerator can be rewritten as:
$$\prod_{a \neq b} (t x_{j_b} - x_{j_a})
= \frac{\prod_{a,b} (t x_{j_b} - x_{j_a})}
{\prod_{a} (t x_{j_a} - x_{j_a})}
= \frac{\prod_{a,b} (t x_{j_b} - x_{j_a})}
{(t-1)^r \prod_{a} x_{j_a}}=\frac{(-1)^{r^2} \prod_{a,b} (x_{j_a} - t x_{j_b})}{(t-1)^r \prod_{a} x_{j_a}}.$$
On the other hand, the denominator can be rewritten as:
$$\prod_{a \neq b} (x_{j_b} - x_{j_a})
= \prod_{a<b} (x_{j_b} - x_{j_a})(x_{j_a} - x_{j_b})
= (-1)^{\binom{r}{2}} \prod_{a<b} (x_{j_a} - x_{j_b})^2.$$
Hence
$$\prod_{\substack{i,k \in J \\ i \neq k}} f_{ik}
= \frac{(-1)^{r^2} \prod_{a,b} (x_{j_a} - t x_{j_b})}
{(t-1)^r (\prod_{a} x_{j_a})\, (-1)^{\binom{r}{2}} \prod_{a<b} (x_{j_a} - x_{j_b})^2}
= (-1)^{r^2 - \binom{r}{2}}
\frac{\prod_{a,b} (x_{j_a} - t x_{j_b})}
{(t-1)^r (\prod_{a} x_{j_a}) \prod_{a<b} (x_{j_a} - x_{j_b})^2}.$$

Note $\displaystyle (-1)^{r^2 - \binom{r}{2}} = (-1)^{r + \binom{r}{2}}$. Now multiply this by $\det(M_J)$ from \eqref{eq: MJ determinant}. The factors combine as:
$$\begin{aligned}
\Bigl(\prod_{i \neq k \in J} f_{ik}\Bigr) \det(M_J) &= \left[(-1)^{r + \binom{r}{2}}
\frac{\prod_{a,b} (x_{j_a} - t x_{j_b})}
{(t-1)^r (\prod_{a} x_{j_a}) \prod_{a<b} (x_{j_a} - x_{j_b})^2}\right]\\
&\times \left[(t-1)^r \, t^{\binom{r}{2}} (-1)^{r + \binom{r}{2}}
\frac{\prod_{a<b} (x_{j_a} - x_{j_b})^2}
{\prod_{a,b} (x_{j_a} - t x_{j_b})}
\;\prod_{a} x_{j_a}\right]\\
&= t^{\binom{r}{2}}.
\end{aligned}$$
and all factors cancel because $\displaystyle (-1)^{2(r+\binom{r}{2})} = 1$.
\end{proof}
\newpage
\bibliographystyle{alpha}
\bibliography{ref.bib}
\end{document}